\documentclass[reqno,12pt]{amsart}

\pdfoutput=1

\usepackage{color} 
\usepackage{ifpdf}
\ifpdf
    \usepackage[pdftex]{graphicx}
    \usepackage[pdftex]{hyperref}
    \hypersetup{
        unicode=false,          
        pdftoolbar=true,        
        pdfmenubar=true,        
        pdffitwindow=false,     
        pdfstartview={FitH},    
        pdftitle={MCP Article},      
        pdfauthor={Michael Holst},   
        pdfsubject={Mathematics},    
        pdfcreator={Michael Holst},  
        pdfproducer={Michael Holst}, 
        pdfkeywords={PDE, analysis, mathematical physics}, 
        pdfnewwindow=true,      
        colorlinks=true,        
        linkcolor=red,          
        citecolor=blue,         
        filecolor=magenta,      
        urlcolor=cyan           
    }

\else
    \usepackage{graphicx}
    \usepackage{pstricks}

\fi

\usepackage{times}
\usepackage{amsfonts}

\usepackage{amsmath}
\usepackage{amsthm}
\usepackage{amssymb}
\usepackage{amsbsy}
\usepackage{amscd}

\usepackage{enumerate}
\usepackage{verbatim}
\usepackage{subfigure}

\newtheorem{theorem}{Theorem}[section]

\newtheorem{lemma}[theorem]{Lemma}

\newtheorem{remark}[theorem]{Remark}

\numberwithin{equation}{section}  

  \newcounter{mnote}
  \let\oldmarginpar\marginpar
    \renewcommand\marginpar[1]{\-\oldmarginpar[\raggedleft\footnotesize #1]%
    {\raggedright\footnotesize #1}}

\definecolor{myblue}{rgb}{0.2,0.2,0.7}
\definecolor{mygreen}{rgb}{0,0.6,0}
\definecolor{mycyan}{rgb}{0,0.6,0.6}
\definecolor{myred}{rgb}{0.9,0.2,0.2}
\definecolor{mymagenta}{rgb}{0.9,0.2,0.9}
\definecolor{mywhite}{rgb}{1.0,1.0,1.0}
\definecolor{myblack}{rgb}{0.0,0.0,0.0}

\newcommand{\beq}{\begin{equation}}
\newcommand{\eeq}{\end{equation}}
\newcommand{\beqa}{\begin{eqnarray}}
\newcommand{\eeqa}{\end{eqnarray}}

\newcommand{\trace}{\mbox{Tr}}
\newcommand{\weakconvg}{\rightharpoonup}

\begin{document}

\title[A Nonlinear Elasticity Model of Conformational Change]
      {A Nonlinear Elasticity Model of Macromolecular Conformational 
       Change Induced by Electrostatic Forces}

\author[Y.C. Zhou]{Yongcheng Zhou}

\author[M. Holst]{Michael Holst}

\author[J.A. McCammon]{James Andrew McCammon}

\email{mholst@math.ucsd.edu}
\address{Department of Mathematics\\
         University of California San Diego\\ 
         La Jolla CA 92093}

\address{Center for Theoretical Biological Physics\\
         Howard Hughes Medical Institute\\ 
         Department of Chemistry and biochemistry\\ 
         University of California San Diego\\ 
         La Jolla CA 92093}

\thanks{MH was supported in
part by NSF Awards 0411723, 022560 and 0511766, in part by DOE
Awards DE-FG02-04ER25620 and DE-FG02-05ER25707, and in part by NIH
Award P41RR08605.}
\thanks{YZ and JAM were supported in part by
the National Institutes of Health, the National Science Foundation,
the Howard Hughes Medical Institute, the National Biomedical
Computing Resource, the National Science Foundation Center for
Theoretical Biological Physics, the San Diego Supercomputing Center,
the W. M. Keck Foundation, and Accelrys, Inc.}

\date{July 22, 2007}


\begin{abstract}
In this paper we propose a nonlinear elasticity model of
macromolecular conformational change (deformation) induced by
electrostatic forces generated by an implicit solvation model.
The Poisson-Boltzmann equation for the
electrostatic potential is analyzed in a domain varying with the
elastic deformation of molecules, and a new continuous model of the
electrostatic forces is developed to ensure solvability of
the nonlinear elasticity equations. We 
derive the estimates of electrostatic forces corresponding to four types 
of perturbations to an electrostatic potential field, and establish the
existance of an equilibrium configuration using a fixed-point argument,
under the assumption that the change in the ionic strength and charges due to
the additional molecules causing the deformation are sufficiently small.
The results are valid for elastic models with arbitrarily complex 
dielectric interfaces and cavities, and can be generalized to large elastic 
deformation caused by high ionic strength, large charges, and strong 
external fields by using continuation methods.
\end{abstract}

\maketitle


{\footnotesize
\tableofcontents
}
\vspace*{-0.5cm}

\section{An Electro-Elastic Model of Conformational Change}
A number of fundamental biological processes rely on the conformational
change of biomolecules and their assemblies. For instance, proteins
may change their configurations in order to undertake new functions,
and molecules may not bind or optimally bind to each other to form
new functional assemblies without appropriate conformational change
at their interfaces or other spots away from binding sites. An
understanding of mechanisms involved in biomolecular conformational
changes is therefore essential to study structures, functions and
their relations of macromolecules. Molecular dynamics (MD)
simulations have proven to be very useful in reproducing the
dynamics of atomistic scale by tracing the trajectory of each atom
in the system \cite{Adcock_MD}. Despite the rapid progress made in
the past decade mainly due the explosion of computer power and
parallel computing, it remains a significant challenge for MD to
study large-scale conformational changes occurring on time-scales
beyond a microsecond \cite{Thorpe_BiophyJ06}. Various coarse-grained
models and continuum mechanics models are developed in this
perspective to complement the MD simulations and to provide
computational tools that are not only able to capture
characteristics of the specific system, but also able to treat large
length and time scales. The prime coarse-grained approaches are the
elastic network models, which describe the biomolecules to be beads,
rods or domains connected by springs or hinges according to the
pre-analysis of their rigidity and the connectivity. Elastic network
models are usually combined with normal mode analysis (NMA) to
extract the dominant modes of motions, and these modes are then used
to explore the structural dynamics at reduced cost \cite{Tama_review06}. Continuum models
do not depend on these rigidity or connectivity analysis. On the
contrary, the rigidity of the structure shall be able to be derived
from the results of the continuum simulations. Typical continuum
models for biomolecular simulations include the elastic deformation
of lipid bilayer membranes \cite{Feng_membrane06} and the gating of
mechanosensitive ion channels \cite{QCui_05channelgating} induced by
given external mechanical loads. It is expected that with more
comprehensive continuum models we will be able to simulate the
variation of the mechanical loads on the macromolecules
with their conformational change, and investigate the dynamics of
molecules by coupling the loads and deformation. This article takes
an important step in this direction by describing and analyzing the
first mathematical model for the interaction between the nonlinear
elastic deformation and the electrostatic potential field of
macromolecules. Such coupled nonlinear models have tremendous
potential in the study of configuration changes and structural
stability of large macromolecules such as nucleic acids, ribosomes
or microtubules during various electrostatic interactions.

Our model is described below. Let $\Omega \in \mathbb{R}^3$ be a smooth open domain whose boundary is noted
as $\partial \Omega$; see Fig.(\ref{fig_domaindef}).
Let the space occupied by the flexible molecules $\Omega_{mf}$ be a smooth subdomain of $\Omega$,
while the space occupied by the rigid molecule(s) is denoted by $\Omega_{mr}$.
Let the remaining space occupied by the aqueous solvent be $\Omega_s$.
The boundaries of $\Omega_{mf}$ and $\Omega_{mr}$ are denoted by $\Gamma_{f}$ and $\Gamma_{r}$, respectively.
We assume that the distance between molecular surfaces and $\partial \Omega$
\begin{eqnarray}
\min \left \{ | x - y | : x \in \Gamma_{f} \cup \Gamma_{r}, y \in \partial \Omega \right \} \label{distance}
\end{eqnarray}
is sufficiently large so that the Debye-H\"uckel approximation can
be employed to determine a highly accurate approximate boundary
condition for the Poisson-Boltzmann equation. There are charges
atoms located inside $\Omega_{mf}$ and $\Omega_{mr}$, and changed mobile ions
in $\Omega_s$. The electrostatic potential field generated by these
charges induces electrostatic forces on the molecules $\Omega_{mf}$ and $\Omega_{mr}$. 
These forces will in turn cause the configuration change of the molecules. We shall model
this configuration rearrangement as an elastic deformation in this study.
Specifically, we will investigate the elastic deformation of
molecule $\Omega_{mf}$ (which is originally in a free state and not
subject to any {\it net} external force) induced by adding molecule
$\Omega_{mr}$ and changing mobile charge density in $\Omega_s$. This
body deformation leads to the displacement of charges in
$\Omega_{mf}$ and the dielectric boundaries, which simultaneously
lead to change of the entire electrostatic potential field. It is
therefore interesting to investigate if the deformable molecule
$\Omega_{mf}$ has a final stable configuration in response to the
appearance of $\Omega_{mr}$ and the change of mobile charge density.
\begin{figure}[!ht] \label{fig_domaindef}
\begin{center}
\includegraphics[width=7cm]{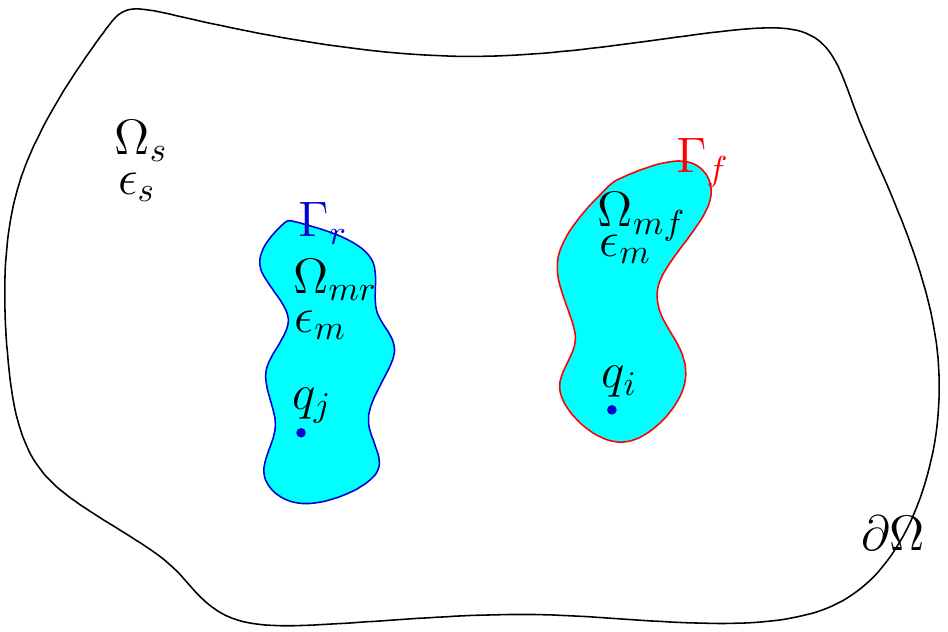}
\caption{Illustration of macromolecules immersed in aqueous solvent environment.}
\end{center}
\end{figure}

Within the framework of an implicit solvent model which treats the
aqueous solvent in $\Omega_s$ as a structure-less dielectric, the
electrostatic potential field of the system is described by the
Poisson-Boltzmann equation (PBE)
\begin{eqnarray}
-\nabla \cdot (\epsilon \nabla \phi) + \kappa^2 \sinh( \phi) = \sum_{i}^{N_f + N_r} q_i \delta(x_i) ~\mbox{in}~\Omega,
\label{PBE_1}
\end{eqnarray}
where $\delta(x_i)$ is the Dirac distribution function at $x_i$,
$N_f + N_r$ is the number of singular charges of the system
including the charges in $\Omega_{mf}$ (i.e. $N_f$) and
$\Omega_{mr}$ (i.e., $N_r$). The dielectric constant $\epsilon$ and
the modified Debye-H\"uckel parameter $\kappa$ are piecewise
constants in domains $\Omega_{mf}, \Omega_{mr}$ and $\Omega_s$. In
particular, $\kappa = 0$ in $\Omega_{mf}$ and $\Omega_{mr}$ because
it models the free mobile ions which appear only in the solvent
region $\Omega_s$. The dielectric constant in the molecule and that
in the solvent are denoted with $\epsilon_m$ and $\epsilon_s$,
respectively. Readers are referred to \cite{Baker_04,Honig_90} for
the importance of the Poisson-Boltzmann equation in biomolecular
electrostatic interactions, and to
\cite{Chen_rpbe,HolstSaied_PBE,HolstBakerWang_PBE,Holst_thesis,HXZhou_96,ZhouMIBPB,Ben_bemPNAS}
for the mathematical analysis as well as various numerical methods
for the Poisson-Boltzmann equation.

The finite(large) deformation of molecules is essential to our coupled model, but
can not be described by a linear elasticity theory. We
therefore describe the displacement vector field $\mathbf{u}(x)$ of
the flexible molecule $\Omega_{mf}$ with a nonlinear elasticity
model:
\begin{equation}
\begin{split} \label{Stru.eq.1}
-\mathrm{div} ( \mathbf{T}(\mathbf{u}) )  = &~ \mathbf{f}_b \quad \mbox{in} \quad \Omega^0_{mf}  \\
\mathbf{T}(\mathbf{u}) \cdot \mathbf{n}  =  &~ \mathbf{f}_s \quad \mbox{on} \quad \Gamma^0_f
\end{split}
\end{equation}
where $\mathbf{f}_b$ is the body force, $\mathbf{f}_s$ is the surface force and $\mathbf{T}(\mathbf{u})$ is the second Piola-Kirchhoff stress tensor.
In this study we assume the macromolecule is a continuum medium obeying the St Venant-Kirchhoff law, and hence its stress tensor is
given by the linear(Hooke's law) stress-strain relation for an
isotropic homogeneous medium:
\begin{eqnarray*}
\mathbf{T}(\mathbf{u}) = (\mathrm{I} + \nabla \mathbf{u}) [
\lambda \trace(\mathbf{E}(\mathbf{u})) \mathrm{I} + 2 \mu \mathbf{E}(\mathbf{u})].
\end{eqnarray*}
Here $\lambda > 0$ and $\mu >0 $ are the Lam\'e constants of the medium, and
\begin{eqnarray*}
\mathbf{E}(\mathbf{u}) = \frac{1}{2} (\nabla \mathbf{u}^{T} + \nabla \mathbf{u} + \nabla \mathbf{u}^{T}  \nabla \mathbf{u})
\end{eqnarray*}
is the nonlinear strain tensor. The equation (\ref{Stru.eq.1}) is
nonlinear due to the Piola transformation ($\mathrm{I} + \nabla
\mathbf{u}$) in $\mathbf{T}(\mathbf{u})$, and the
quadratic term in the nonlinear strain $\mathbf{E}(\mathbf{u})$. The
third potential source of nonlinearity, namely a nonlinear
stress-strain relation, is not considered here; however, our methods
apply to this case as well.

It is noted that Eq.(\ref{Stru.eq.1}) is defined in the undeformed molecule body $\Omega^0_{mf}$ with
undeformed boundary $\Gamma^0_{f}$, while the Poisson-Boltzmann equation (\ref{PBE_1}) holds true for real 
deformed configurations. The deformed configuration is unknown before we solved the coupled system.  
We therefore define a displacement-dependent mapping $\Phi(\mathbf{u})(x):\Omega^0 \to \Omega$ and apply
this mapping to the the Poisson-Boltzmann equation such that it can also be analyzed on the undeformed molecular configuration. In $\overline{\Omega}_{mf}$ this map $\Phi(\mathbf{u})(x)$ is $\mathcal{I} + \mathbf{u}$ where
$\mathcal{I}$ is the identity mapping. A key technical tool in our work is that this mapping
is then harmonically extended to $\Omega$ to obtain
the maximum smoothness. Apply this mapping, the Poisson-Boltzmann equation \ref{PBE_1} changes to be
\begin{eqnarray}
-\nabla \cdot (\epsilon \mathbf{F}(\mathbf{u}) \nabla \phi) + J(\mathbf{u}) \kappa^2 \sinh( \phi) =
\sum_{i}^{N_f + N_r} J(\mathbf{u}) q_i \delta(\Phi(x) - \Phi(x_i)) ~\mbox{in}~\Omega, \label{PBE_2}
\end{eqnarray}
where $J(\mathbf{u})$ is the Jacobian of $\Phi(\mathbf{u})$ and
\begin{eqnarray}
\mathbf{F}(\mathbf{u}) = (\nabla \Phi(\mathbf{u}) )^{-1} J(\mathbf{u}) (\nabla \Phi(\mathbf{u}) )^{-T}. \label{Fu_def}
\end{eqnarray}
This matrix $\mathbf{F}$ is well defined whenever $\Phi(\mathbf{u})$ is a $C^1$-diffeomorphism \cite{Grandmont_3Dcouple}.
The functions in Eq.(\ref{PBE_2}) should be interpreted as the compositions of respective functions in Eq.(\ref{PBE_1})
with mapping $\Phi(x)$, i.e., $\phi(x) = \phi(\Phi(x)), \epsilon(x) = \epsilon(\Phi(x))$ and $\kappa(x) = \kappa(\Phi(x))$.

In this paper, we shall analyze the existence of the coupled
solution of the elasticity equation (\ref{Stru.eq.1}) and the
transformed Poisson-Boltzmann equation (\ref{PBE_2}). These two
equations are coupled through displacement mapping
$\Phi(\mathbf{u})$ in the Poisson-Boltzmann equation and the
electrostatic forces to be defined later. The solution of this
coupled system represents the equilibrium between the elastic stress
of the biomolecule and the electrostatic forces to which the
biomolecule is subjected. The existence, the uniqueness and the
$W^{2,p}$-regularity of the elasticity solution have already been
established by Grandmont \cite{Grandmont_3Dcouple} in studying the
coupling of elastic deformation and the Navier-Stokes equations;
thus in this work we shall focus on the solution to the transformed
Poisson-Boltzmann equation and to the coupled system.
We shall define a mapping $S$ from an appropriate space $X_p$ of displacement field $\mathbf{u}$ into itself,
and seek the fixed-point of this map. This fixed-point, if it exists, will be the solution of the coupled system.
A critical step in defining $S$ is the harmonic extension of the Piola transformation
from $\Omega_{mf}$ to $\Omega$ and $\mathbb{R}^3$. The regularity of the Piola transformation determines not only the existence of the solution to the transformed Poisson-Boltzmann equation, but also the existence of the solution to the coupled system.
Because most of our analysis will be carried out on the undeformed configuration we will still use $\Omega_{mf}, \Omega_{mr}, \Omega_s, \Gamma_f, \Gamma_r$ to denote the undeformed configurations of molecules, the solvent and the molecular interfaces, unless otherwise specified.

The paper is organized as follows. In Section \ref{Prelim} we review
a fundamental result concerning the piecewise $W^{2,p}$-regularity
of the solutions to elliptic equations in non-divergence form and
with discontinuous coefficients. The nonlinear elasticity equation
will be discussed in Section \ref{ElasticEq}, where the major
results from \cite{Grandmont_3Dcouple} are presented without proof.
The Piola transformation will be defined, harmonically extended, and
then analyzed. In Section \ref{PBEsolution} we will prove the
existence and uniqueness of the solution to the Piola-transformed
Poisson-Boltzmann equation, generalizing the results in
\cite{Chen_rpbe} for the un-transformed case. Both $L^{\infty}$ and
$W^{2,p}$ estimates will be given for the electrostatic potential in
the solvent region, again generalizing results in \cite{Chen_rpbe}.
We will then define the electrostatic forces and estimate
these forces by decomposing them into components corresponding to four
independent perturbation steps. The estimates of these components
are obtained separately and the final estimate of the surface force
is assembled from these individual estimates. The coupled system
will be finally considered in Section \ref{CoupledSystem} where the
mapping $S$ will be defined, and the main result of the paper will
be established by applying a fixed-point theorem on this map to give
the existence of a solution of the coupled system.

\section{Notation and Some Basic Estimates} \label{Prelim}
In what follows $W^{k,p}(\mathcal{D})$ will denote the standard
Sobolev space on an open domain $\mathcal{D}$, where $\mathcal{D}$
can be $\Omega, \Omega_{m}$ or $\Omega_s$. While solutions of the
Poisson-Boltzmann have low global regularity in $\Omega$, we will
need to explore and exploit the optimal regularity of the solution
in any subdomain of $\Omega$. For this purpose, we define
$\mathcal{W}^{2,p}(\Omega) = W^{2,p}(\Omega_m) \dotplus
W^{2,p}(\Omega_s)$ where $\dotplus$ is the direct sum. Every
function $\phi \in \mathcal{W}^{2,p}$ can be written as $\phi(x) =
\phi_m(x) + \phi_s(x)$ where $\phi_m(x) \in W^{2,p}(\Omega_m),
\phi_s(x) \in W^{2,p}(\Omega_s)$, and has a norm
\begin{eqnarray}
\| \phi \|_{\mathcal{W}^{2,p}} = \| \phi_m \|_{W^{2,p}(\Omega_m)} + \| \phi_s \|_{W^{2,p}(\Omega_s)}. \label{w2p_norm}
\end{eqnarray}
Similarly, we define a class of functions $\mathcal{C}=\mathcal{C}(\Omega)$
which are continuous in either subdomain and may have finite jump on the interface, i.e., a
function $a \in \mathcal{C}$ is given by $a=a_m + a_s$ where $a_m \in C(\Omega_m), a_s \in C(\Omega_s)$
are continuous functions in their respective domains. The norm in $\mathcal{C}$ is defined by
$$\|a \|_{\mathcal{C}} = \|a_m \|_{C(\Omega_m)} + \|a_s \|_{C(\Omega_s)}.$$

We recall two important results. The first is a technical lemma which will be used for the
estimation of the product of two $W^{1,p}$ functions; this is sometimes called the Banach
algebra property.
\begin{lemma} \label{w1p_algebra}
Let $ 3<p<\infty, 1 \le q \le p$ be two real numbers. Let $\Omega$
be a domain in $\mathbb{R}^3$. Let $u \in W^{1,p}(\Omega), v \in
W^{1,q}(\Omega)$, then their product $uv$ belongs to $W^{1,q}$, and
there exists a constant $C$ such that
$$ \| uv \|_{W^{1,q}(\Omega)} \le C \| u \|_{W^{1,p}(\Omega)} \| v \|_{W^{1,q}(\Omega)}.$$
\end{lemma}
For the proof of this lemma we refer to \cite{Adams_SobolevBook}. In
this paper we will apply Lemma (\ref{w1p_algebra}) to the case with
$p=q$. The second result is a theorem concerning the $L^p$ estimate
of elliptic equations with discontinuous coefficients.
\begin{theorem} \label{w2p_theorem}
Let $\Omega$ and $ \Omega_1 \subset \subset \Omega$ be bounded domains of $\mathbb{R}^3$ with smooth boundaries $\partial \Omega$ and $\Gamma$. Let $\overline{\Omega}_1 =(\Omega_1 \cup \Gamma)$ and
$\Omega_2 = \Omega \setminus \overline{\Omega}_1$.
Let $A$ be a second order elliptic operator such that
\begin{eqnarray*}
(Au)(x) =
\begin{cases}
(A_1 u)(x)  & x \in \Omega_1 \\
(A_2 u)(x)  & x \in \Omega_2
\end{cases}, ~\mbox{where}~ A_i = \sum_{k \le2} a_{ik}(x) D^{k}.
\end{eqnarray*}
Then there exists a unique solution $u \in \mathcal{W}^{2,p}$ for the interface problem
\begin{eqnarray*}
(Au)(x) & = & f ~ \mbox{in} ~ \Omega  \\
  \left [ u \right] = u_2 - u_1 & = & 0 ~ \mbox{on}~ \Gamma  \\
  \left [ B u_n \right ]  = B_2 \nabla u_2 \cdot \mathbf{n} - B_1 \nabla u_1 \cdot \mathbf{n} & = & h ~ \mbox{on}~ \Gamma  \\
  u & = & g ~ \mbox{on}~ \partial \Omega
\end{eqnarray*}
providing that $a_{ik} \in \mathcal{C}(\Omega),
B_i \in C(\Gamma), f \in L^p(\Omega), g \in W^{2-1/p,p}(\partial \Omega),
 h \in W^{1-1/p,p}(\Gamma)$, where $\mathbf{n}$ is the outside normal to $\Omega_1$. Moreover, the following estimate holds true
\begin{eqnarray}
\hspace*{-0.4cm}
\| u \|_{\mathcal{W}^{2,p}(\Omega)} \le K \left (  \| f \|_{L^p(\Omega)} + \| h \|_{W^{1-1/p,p}(\Gamma)}
+ \| g \|_{W^{2-1/p,p}(\partial \Omega)} + \| u \|_{L^p(\Omega)}  \right ), \label{w2p_estimate}
\end{eqnarray}
where the constant $K$ depends only on $\Omega, \Omega_1, \Omega_2, p$ and the modulus of continuity of $A$.
\end{theorem}
Theorem (\ref{w2p_theorem}) is fundamental to various results about
elliptic equations with discontinuous coefficients;
For example, the global $H^1$ regularity and $H^2$ estimates of
Bab\"uska \cite{Babuska_70}, the finite element approximation of Chen
{\it et al.} \cite{ChenZM_98}, {\it a prior} estimates for second-order elliptic interface
problems \cite{ZouJ_02},
the solution theory and estimates for the nonlinear
Poisson-Boltzmann equation \cite{Chen_rpbe},
and the continuous and discrete {\em a priori} $L^{\infty}$ estimates
for the Poisson-Boltzmann equation along with a quasi-optimal
{\em a priori} error estimate for Galerkin methods \cite{Chen_rpbe}
applied to the Poisson-Boltzmann equation.
For the proof of Theorem (\ref{w2p_theorem}) and the more general conclusions
for high-order elliptic equations with high-order interface conditions we
refer to \cite{Seftel_64,Seftel_65}.

\section{Nonlinear Elasticity and the Piola Transformation} \label{ElasticEq}
We first state a theorem concerning the existence, uniqueness, regularity and the estimation of
the solution to the nonlinear elasticity equation \cite{Grandmont_3Dcouple}:
\begin{theorem} \label{thm_3Delasticity}
Let the body force $\mathbf{f}_b \in L^p(\Omega_{mf}) $ and the surface
force $\mathbf{f}_s \in W^{1-1/p,p}(\Gamma_f)$, where $3 < p <
\infty$. There exists a neighborhood of 0 in $L^p(\Omega_{mf})
\times W^{1-1/p,p}(\Gamma_f)$ such that if $(\mathbf{f}_b,
\mathbf{f}_s)$ belongs to this neighborhood then there exists a
unique solution $\mathbf{u} \in W^{2,p}(\Omega_{mf}) \cap
W^{1,p}_{0,\Gamma_{f0}}(\Omega_{mf})$ of
\begin{equation} \label{Eq_elasticity}
\begin{split}
-\mathrm{div} ( \mathbf{T} (\mathbf{u}) )  = &~ \mathbf{f}_b ~ \mbox{in} \quad \Omega_{mf}, \\
\mathbf{T}(\mathbf{u}) \mathbf{n}   = & ~\mathbf{f}_s ~ \mbox{on} \quad \Gamma_{f} \setminus \Gamma_{f0}, \\
\mathbf{u}   = &~ 0 ~ \mbox{on} \quad \Gamma_{f0},  \\
\int_{\Gamma_f} (\mathrm{I} + \nabla \mathbf{u}) J(\mathbf{u}) (\mathrm{I} + \nabla \mathbf{u})^{-T} \cdot \mathbf{n}  = &~
3 |\Omega_{mf}|,
\end{split}
\end{equation}
where $\Gamma_{f0}$ is a subset of $\Gamma_f$ equipped with homogeneous Dirichlet boundary condition,
$\mathrm{I}$ is the unit matrix. The last equation represents the incompressibility condition of
the elastic deformation. Moreover, the solution can be estimated with respect to the force data:
\begin{eqnarray}
\| u\|_{W^{2,p}(\Omega_{mf})} \le C(\|\mathbf{f}_b\|_{L^p(\Omega_{mf})} + \|\mathbf{f}_s\|_{W^{1-1/p,p}(\Gamma_f)}).  \label{estimate_u}
\end{eqnarray}
\end{theorem}
\begin{proof}
See \cite{Ciarlet_book}. \qed
\end{proof}
\begin{remark}
It is noticed that $\mathbf{u} \in C^{1,1-3/p}(\overline{\Omega}_{mf})$ because of the continuous embedding of
$W^{2,p}(\Omega_{mf})$ in $C^{1,1-3/p}(\overline{\Omega}_{mf})$ for $p > 3$.
\end{remark}

The displacement field $\mathbf{u}(x)$ solved from Eq.(\ref{Eq_elasticity}) naturally defines a mapping $\Phi(\mathbf{u}) =
\mathcal{I} + \mathbf{u}$ in $\overline{\Omega}_{mf}$ where $\mathcal{I}$ is the identity mapping. This mapping
$\Phi(\mathbf{u})(x)$ has to be appropriately extended into $\mathbb{R}^3 \setminus \overline{\Omega}_{mf}$ to yield a global transformation for the Poisson-Boltzmann equation. It is critical in what follows that this extension has various
favorable properties, which leads us to define a global mapping by harmonic extension:
\begin{eqnarray}
\Phi(\mathbf{u}) = \begin{cases}
                   \mathcal{I} + \mathbf{u} &  \mathbf{x} \in \overline{\Omega}_{mf} \\
                   \mathcal{I} + \mathbf{w} &  \mbox{otherwise} \\
                   \end{cases}  \label{Phi_def}
\end{eqnarray}
where $\mathbf{w}$ solves
\begin{equation} \label{w_def}
\begin{split}
\Delta \mathbf{w}  = &~ 0  ~\mbox{in}~ \mathbb{R}^3 \setminus \overline{\Omega}_{mf} , \\
\mathbf{w} = & ~\mathbf{u} ~  \mbox{on} ~\Gamma_{f}. \\
\end{split}
\end{equation}
The following crucial lemma concerns the regularity of $\Phi(\mathbf{u})$ and the invertibility of $\nabla \Phi(\mathbf{u})$:
\begin{lemma} \label{piola}
Let $\Phi(\mathbf{u})$ be defined in Eq.(\ref{Phi_def}), we have
\begin{itemize}
\item[(a)] $\Phi(\mathbf{u}) \in W^{2,p}(\Omega_{mf})$ and $\Phi(\mathbf{u}) \in
C^{\infty}(\mathbb{R}^3 \setminus \overline{\Omega}_{mf})$.
\item[(b)] There exists a constant $M>0$ such that for all $\| \mathbf{u} \|_{W^{2,p}(\Omega)} \le M$,
$\nabla \Phi(\mathbf{u})$ is an invertible matrix in $W^{1,p}(\Omega_{mf})$ and in
$C^{\infty}(\mathbb{R}^3 \setminus \overline{\Omega}_{mf})$.
\item[(c)] Under condition (b), $\Phi(x)$ is one-to-one on $\mathbb{R}^3$, and is a $C^1$-diffeomorphism from
$\Omega_{mf}$ to $\Phi(\mathbf{u})(\Omega_{mf})$, and is a $C^{\infty}$-diffeomorphism
from $\mathbb{R}^3 \setminus \overline{\Omega}_{mf}$ to $\Phi(\mathbf{u})(\mathbb{R}^3 \setminus \overline{\Omega}_{mf})$.
\end{itemize}
\end{lemma}
\begin{proof}
That $\Phi(\mathbf{u}) \in W^{2,p}(\Omega_{mf})$ follows directly from
its definition.  Also, of $\Phi(\mathbf{u}) \in C^{\infty}(\mathbb{R}^3
\setminus \overline{\Omega}_{mf})$ since $\Phi(\mathbf{u}) =
\mathcal{I} + \mathbf{w}$ while $\mathbf{w}$ is harmonic hence
analytical in $\Phi(\mathbf{u}) \in C^{\infty}(\mathbb{R}^3
\setminus \overline{\Omega}_{mf})$ because it is the solution of the
Laplace equation (\ref{w_def}). For the invertibility of
$\Phi(\mathbf{u})$ in $W^{1,p}(\Omega_{mf})$ we refer to Lemma 2 in
\cite{Grandmont_3Dcouple} or Theorem 5.5.1 in \cite{Ciarlet_book},
which says that if a $\mathbf{u} \in \Omega_{mf} $ is differentiable
and
$$ | \nabla \mathbf{u}(x) | < C $$
for some constant depending on $\Omega_{mf}$, then $\nabla \Phi(\mathbf{u}) = \mathrm{I} + \nabla \mathbf{u} > 0~ \forall~
x \in \overline{\Omega}_{mf}$ and $\mathrm{I} + \nabla \mathbf{u}$ is injective on $\Omega_{mf}$.
The invertibility of $\nabla \Phi(\mathbf{u})$
therefore follows from the facts that $\mathbf{u} \in C^{1,1-3/p}(\overline{\Omega}_{mf})$ such that
for sufficiently small $M$
$$|\nabla \mathbf{u}| \le \| \mathbf{u} \|_{C^{1,1-3/p}(\Omega_{mf})} \le C_1 \| \mathbf{u} \|_{W^{2,p}(\Omega_{mf})}
= C_1 M \le C.$$
To prove the invertibility of $\nabla \Phi(\mathbf{u}) = \mathrm{I} + \nabla \mathbf{w}$ in
$\mathbb{R}^3 \setminus \overline{\Omega}_{mf} $ we
notice the following estimate for the first derivative of the solution to Laplace equation \cite{GT_pdebook}:
$$ | \nabla \mathbf{w} | \le \| \mathbf{w} \|_{C^{1}(\mathbb{R}^3 \setminus \overline{\Omega}_{mf})} \le C_2 \| \mathbf{u} \|_{C^{1,1-3/p}(\Gamma_f)} \le C_2 M \le C,$$
Therefore if $M$ is chosen such that
\begin{eqnarray}
 M \le \frac{C}{\max \{ C1,C2 \} } \label{M_bound_1}
\end{eqnarray}
$\nabla \Phi(\mathbf{u})$ is an invertible matrix in $\mathbb{R}^3$. \qed
\end{proof}
\begin{remark}
It follows from Lemma (\ref{piola}) that the matrix $\mathbf{F}(\mathbf{u})$ in Eq.(\ref{Fu_def})
is well-defined, symmetric and positive definite. More precisely,
we have that the maps
$\mathbf{F}(\mathbf{u})(x) \in C^{0,1-3/p}(\overline{\Omega}_{mf})$ and $\mathbf{F}(\mathbf{u})(x) \in
C^{\infty}(\mathbb{R}^3 \setminus \overline{\Omega}_{mf})$. On the other hand, as a mapping from $\mathbf{u} \in W^{2,p}(\Omega_{mf})$ to
$\mathbf{F}(\mathbf{u}) \in C^{\infty}(\mathbb{R}^3 \setminus \overline{\Omega}_{mf})$,
$\mathbf{F}(\mathbf{u})$ is infinitely differentiable with respect to $\mathbf{u}$. In all what follows
we will write $\mathbf{F}(\mathbf{u})$ and $J(\mathbf{u})$ as $\mathbf{F}$ and $J$ only,
keeping in mind that they are $\mathbf{u}$ dependent.
\end{remark}

\section{Preliminary Results for the Poisson-Boltzmann Equation} \label{PBEsolution}

\subsection{The Poisson-Boltzmann equation with Piola transformation}
The rigorous analysis and numerical approximation of solutions to the
Poisson-Boltzmann equation (\ref{PBE_1}) or its transformed
version (\ref{PBE_2}) are generally subject to three major difficulties:
1) the singular charge distribution,
2) the discontinuous dielectric constant on the molecular surface and
3) the strong exponential nonlinearities.
However, it was recently demonstrated \cite{Chen_rpbe} that as far as the
untransformed Poisson-Boltzmann equation (\ref{PBE_1}) is concerned,
some of these difficulties can be side-stepped by individually considering
the singular and the regular components of the solution.
Specifically, the potential solution is decomposed to be
\begin{eqnarray}
\phi = G + \phi^r = G + \phi^l + \phi^n \label{G_phir_decomp}
\end{eqnarray}
where the singular component
$$ G = \sum_i \frac{q_i}{\epsilon_m |x - x_i|}$$
is the solution of the Poisson equation
\begin{eqnarray}
-\nabla \cdot (\epsilon_m \nabla G) = \rho_f := \sum_i^{N} q_i \delta(x_i) \quad \mbox{in} \quad \mathbb{R}^3; \label{G_utransd}
\end{eqnarray}
while $\phi^l$ is the linear component of the electrostatic potential which satisfies
\begin{equation} \label{pl_utransd}
\begin{split}
-\nabla \cdot (\epsilon \nabla \phi^l) = & ~
-\nabla \cdot ((\epsilon - \epsilon_m) \nabla G) \quad \mbox{in} \quad \Omega,  \\
\phi^l = & ~ g - G \quad \mbox{on} \quad  \partial \Omega,
\end{split}
\end{equation}
and the nonlinear component $\phi^n$ solves
\begin{equation} \label{pn_utransd}
\begin{split}
-\nabla \cdot (\epsilon \nabla \phi^n) + \kappa^2 \sinh(\phi^n + \phi^l + G)  = &~ 0  \quad \mbox{in} \quad \Omega,  \\
\phi^n  = & ~0 \quad \mbox{on} \quad \partial \Omega,
\end{split}
\end{equation}
where
\begin{eqnarray}
 g = \sum_{i=1}^{N} q_i \frac{e^{-\kappa |x - x_i|}}{\epsilon_s |x - x_i|} \label{g_def}
\end{eqnarray}
is the boundary condition of the complete Poisson-Boltzmann equation
(\ref{PBE_1}). Such a decomposition scheme removes the point charge
singularity from the original Poisson-Boltzmann and it was shown in
\cite{Chen_rpbe} that the regular component of the electrostatic
potential $\phi^r = \phi^l + \phi^n$ belongs to $H^1(\Omega)$ although
the entire solution $G + \phi^r$ does not. The most prominent
advantage of this decomposition lies in the fact that the regular
component represents the reaction potential field of the system,
which can be directly used to compute the solvation energy and other
associated important properties of the system. It is not necessary
to solve the Poisson-Boltzmann equation twice, once with uniform
vacuum dielectric constant and vanishing ionic strength and the
other with real physical conditions, to obtain the reaction field
\cite{HXZhou_96}. As to be shown later on, the identification of
this regular potential component as the reaction field also
facilitates the analysis and the computation of the electrostatic
forces.

Applying the similar decomposition to the transformed
Poisson-Boltzmann equation we get an equation for the singular
component $G$:
\begin{eqnarray}
-\nabla \cdot (\epsilon_m \mathbf{F} \nabla G) = J \rho_f \quad \mbox{in} \quad \mathbb{R}^3, \label{G_transd}
\end{eqnarray}
and an equation for the regular component $\phi^r$:
\begin{equation} \label{pr_transd}
\begin{split}
-\nabla \cdot (\epsilon \mathbf{F} \nabla \phi^r) + J \kappa^2 \sinh(\phi^r + G)  = & ~
\nabla \cdot ( (\epsilon - \epsilon_m) \mathbf{F} \nabla G)  \quad \mbox{in} \quad \Omega,  \\
\phi^r  = &~ g - G \quad \mbox{on} \quad  \partial \Omega.
\end{split}
\end{equation}
We shall prove the existence of $\phi^r$ in Eq.(\ref{pr_transd}) and
give its $L^{\infty}$ bounds by individually considering the
equation for the linear component $\phi^l$:
\begin{equation}
\begin{split} \label{pl_transd}
-\nabla \cdot (\epsilon \mathbf{F} \nabla \phi^l) = & ~
\nabla \cdot ( (\epsilon - \epsilon_m) \mathbf{F} \nabla G)  \quad \mbox{in} \quad \Omega,  \\
\phi^l  = &~ g - G \quad \mbox{on} \quad  \partial \Omega,
\end{split}
\end{equation}
and the equation for the nonlinear component $\phi^n$:
\begin{equation} \label{pn_transd}
\begin{split}
-\nabla \cdot (\epsilon \mathbf{F} \nabla \phi^n) + J \kappa^2 \sinh(\phi^n + \phi^l + G)  =
& ~0  \quad \mbox{in} \quad \Omega,  \\
\phi^n  = &~ 0 \quad \mbox{on} \quad \partial \Omega.
\end{split}
\end{equation}
As mentioned above, the functions $G,\phi^l,\phi^n,\rho^f$ and $\kappa$ in Eqs.(\ref{G_transd}) through (\ref{pn_transd})
shall be interpreted as the compositions of the corresponding entries of these functions in untransformed equations
(\ref{G_utransd}) through (\ref{pn_utransd}) with the Piola transformation $\Phi(x)$, i.e., $g=g(\Phi(x)),G=G(\Phi(x)),
\phi^l = \phi^l(\Phi(x)),\phi^n=\phi^n(\Phi(x)),\rho^f=\rho^f(\Phi(x)), \kappa=\kappa(\Phi(x))$.

\subsection{Regularity and estimates for the singular solution component $G$}
We first study the Eq. (\ref{G_transd}) for the singular component of electrostatic potential. We remark that
the linear and nonlinear PB equations have the same singular component of the electrostatic potential. The solution
of this singular component is the Green's function for the elliptic operator $L$ defined by
\begin{eqnarray}
 L u = -\nabla \cdot (\epsilon_m \mathbf{F} \nabla u). \label{green_operator}
\end{eqnarray}
We shall use the following theorem \cite{Gruter_greefunc} concerning the regularity and the estimate of the Green's function:
\begin{theorem} \label{green_estimate}
Let $\Omega$ be an open set in $\mathbb{R}^3$. Suppose the elliptic operator
$$ L u = \sum_{i,j=1}^{n} \frac{\partial}{\partial x_j}(a_{ij} \frac{\partial u}{\partial x_i})$$
is uniformly elliptic and bounded, while the coefficients $a_{ij}$ satisfying
$$ | a_{ij}(x) - a_{ij}(y)| \le \omega(|x - y|) $$
for any $x,y \in \Omega $, and the non-decreasing function $\omega(x)$
satisfies
\begin{eqnarray*}
\omega(2t) & \le & K \omega(t)~~ \mbox{for some}~ K > 0~\mbox{and all}~t>0, \\
\int_{\mathbb{R}} \frac{\omega(t)}{t}  dt & <  &\infty.
\end{eqnarray*}
Then for the corresponding
Green's function $G$ the following six inequalities are true for any $x,y \in \Omega$:
\begin{itemize}
\item[(a)] $G(x,y) \le K |x - y |^{-1}$,
\item[(b)] $G(x,y) \le K \delta(x) |x - y |^{-2}$.
\item[(b)] $G(x,y) \le K \delta(x) \delta(y)
  |x - y |^{-3}$.
\item[(d)] $|\nabla_{x} G(x,y)| \le K  |x - y |^{-2}$.
\item[(e)] $|\nabla_{y} G(x,y)| \le K \delta(y)
|x - y |^{-3}$.
\item[(f)] $|\nabla_{x} \nabla_{y} G(x,y)| \le K
|x - y |^{-3}$.
\end{itemize}
where $\delta(y) = \mathrm{dist}(y,\partial \Omega)$ and the general constant $K=K(a_{ij},\omega,\Omega)$.
\end{theorem}
From this theorem we can derive the regularity of the Green's
function of the operator (\ref{green_operator}). Indeed, by Sobolev
embedding $\epsilon_m \mathbf{F} \in C^{0,1-3/p}(\mathbb{R}^3)$,
therefore it satisfies the conditions on $a_{ij}$ in this theorem
provided that $\omega(t) = Kt^{3/p}$. We then conclude that the
singular component of the electrostatic potential $G \in
W^{1,\infty}(\Omega \setminus B_r(x_i))$. On the other hand, from
Eq. (\ref{G_transd}) we know that $G(\Phi(\mathbf{u})(x))/J(x_i)$
itself is the Green's function of operator (\ref{green_operator}) if
$\mathbf{F}$ is generated by the Piola transformation according to
(\ref{Fu_def}) and $J$ is the corresponding Jacobian. Thus the
Green's function of differential operator (\ref{green_operator})
belongs to $W^{2,p}(\Omega \setminus B_r(x_i))$ since it is the
composition of the Green's function of Laplace operator, which is of
$C^{\infty}(\Omega \setminus B_r(x_i))$, and the Piola
transformation, which is of $W^{2,p}(\Omega)$. Higher regularity of
$G$ in $\Omega_s$ can be derived thanks to the harmonic extension of
$\mathbf{u}$ to $\mathbb{R}^3 \setminus \overline{\Omega}_{mf}$. In
particular, because all charges are located in $\Omega_{mf}$ and
$\Omega_{mr}$ the Poisson equation (\ref{G_transd}) appears a
Laplace equation
\begin{eqnarray*}
\nabla (\epsilon \mathbf{F} \nabla G ) = 0 \quad \mbox{in} \quad \Omega_s, \label{G_in_s}
\end{eqnarray*}
hence $G(x) \in C^{\infty}(\Omega_s)$, since $\Omega_s$ is a smooth open domain and
$\mathbf{F} \in C^{\infty}(\Omega_s)$.

In addition to the regularity of the Green's function, we have following estimates of $G$ with respect to $\mathbf{F}$ and $J$.
\begin{lemma} \label{lemma_green}
For any given molecule the Green's function $G$ of operator (\ref{green_operator}) has estimates
\begin{itemize}
\item[(a)] $\| G \|_{L^{\infty}(\overline{\Omega}_s)}  \le  C \|J\|_{L^{\infty}(\Omega)} $.
\item[(b)] $\| \nabla G \|_{L^{\infty}(\overline{\Omega}_s)}  \le  C \|J \|_{L^{\infty}(\Omega)}$.
\end{itemize}
If in addtion $\|\mathbf{F} -\mathrm{I} \|_{W^{1,p}(\Omega)} \le
C_f, \|J - 1 \|_{W^{1,p}(\Omega)} \le C_J$ for some constant $C_f$
and $C_J$, then
\begin{itemize} \label{g_G_estimate}
\item[(c)] $\| G \|_{L^{p}(\partial \Omega)}   \le  C \| G \|_{L^{\infty}(\overline{\Omega}_s)} $.
\item[(d)] $\| g \circ \Phi \|_{W^{2-1/p,p}(\partial \Omega)} \le
C_g  \|g \|_{W^{2,p}(\Omega_s)}$.
\item[(e)] $\| g \circ \Phi - G \|_{W^{2-1/p,p}(\partial \Omega)}  \le C_g  \|g \|_{W^{2,p}(\Omega_s)} + C_G \| G \|_{L^{\infty}(\Omega_s)}   $.
\item[(f)] $\| \mathbf{F} \nabla G \|_{W^{1-1/p,p}(\Gamma)}  \le C_{\Gamma} \| G\|_{L^{\infty}(\Omega'_s)}$ for some set $\Omega'_s$.
\end{itemize}
\end{lemma}
\begin{proof}
This $\| J \|_{L^{\infty}(\overline{\Omega}_s)}$ is well defined
since $J$ is uniformly continuous in $\overline{\Omega}_s$. To
prove (a) and (b) we define $q_{max} = \mbox{max}\{ |q_i| \}$ and
$$ \displaystyle{
\| \nabla_{x} G_i(x ,x_i) \|_{L^{\infty}(\overline{\Omega}_s)}  =
\frac{K}{\delta^2}, ~~ \|  G_i(x ,x_i)
\|_{L^{\infty}(\overline{\Omega}_s)}  = \frac{K}{\delta}}$$ where
$\delta$ is the smallest distance between $x \in \partial \Omega$
and singular charges at $x_i$. This smallest distance is related to
the radii of atoms used in defining the molecular surface. In the
sense of Connolly's molecular surface, $\delta$ is simply the
smallest van der Waals radius of the atoms which have contact
surface \cite{Connolly_MS}. We can therefore bound $G$ and its
gradient with
\begin{eqnarray}
\| G \|_{L^{\infty}(\overline{\Omega}_s)} & = &
 \| \sum_i  J q_i G_i \|_{L^{\infty}(\overline{\Omega}_s)} \le  N q_{max} \| J\|_{L^{\infty}(\overline{\Omega}_s)}
\| G_i \|_{L^{\infty}(\overline{\Omega}_s)} \nonumber \\
& =  & \frac{\|J\|_{L^{\infty}(\Omega)} N K q_{max}}{\delta}, \label{green_estim} \\
\| \nabla G \|_{L^{\infty}(\overline{\Omega}_s)}  & = & \| \sum_i J q_i
\nabla G_i \|_{L^{\infty}(\overline{\Omega}_s)} \le  N q_{max} \| J
\|_{L^{\infty}(\overline{\Omega}_s)} \| \nabla_x  G_i
\|_{L^{\infty}(\overline{\Omega}_s)} \nonumber \\
& = &  \frac{\|J
\|_{L^{\infty}(\Omega)} N K q_{max}}{\delta^2},
\label{green_grad_estim}
\end{eqnarray}
where $N$ is the total number of singular charges and $\| J \|_{L^{\infty}(\overline{\Omega}_s)}$ is
the maximum Jacobian on $\Gamma$.

The statement $(c)$ holds because $\partial \Omega$ is also a piece
of boundary of $\Omega_s$ as shown in Fig.(\ref{fig_domaindef}). To
verify the statement $(d)$, we noted that $g \circ \Phi $ is the
composition of $g$ in Eq.(\ref{g_def}), which is smooth in
$\Omega_s$, and the mapping $\Phi(x) \in W^{2,p}(\Omega_s)$, i.e.,
$$ g \circ \Phi = \sum_{i} q_i \frac{e^{-\kappa |\Phi(x) - \Phi(x_i)|}}{\epsilon_s | \Phi(x) - \Phi(x_i) |}. $$
Following the estimate of the composite function in Sobolev space \cite{Moseenkov_sobolevcompos},
we have the inequality
\begin{eqnarray}
\| g \circ \Phi \|_{W^{2-1/p,p}(\partial \Omega)} & \le & \| g \|_{W^{2,p}(\Omega_s)} \nonumber \\
& \le & C(1+\| \Phi \|_{L^{\infty}(\Omega_s)}) (1 + \| \Phi \|_{W^{2,p}(\Omega_s)}) \| g \|_{W^{2,p}(\Omega_s)}
\nonumber \\
& := & C_g \| g \|_{W^{2,p}(\Omega_s)} \label{g_est}
\end{eqnarray}
with a constant $C_g$ depending upon $\Phi(x)$. Here we choose to bound $\| g \circ \Phi \|_{W^{2-1/p,p}(\partial \Omega)}$
by $\| g \circ \Phi \|_{W^{2,p}(\Omega_s)}$ instead of $\| g \circ \Phi \|_{W^{2,p}(\Omega)}$ since the latter is not well
defined due to the singular nature of $g$.

The validity of inequalities $(e)$ and $(f)$ follows from the estimate of $\| G \|_{W^{2,p}(\Omega_s^{'})}$.
This $\Omega_s^{'}$ is chosen such that $\Omega_s \subset \subset \Omega_s^{'}$. For example, we can choose
$\Omega_s^{'}$ to be the union of $\Omega_s$, $\Gamma$, $\partial \Omega$, the domain
\begin{eqnarray*}
 \Omega^-_s= \{ x | x \in \Omega_{mf}, \mathrm{dist}(x,\Gamma) < \frac{\delta}{2} \},
\end{eqnarray*}
and the domain
\begin{eqnarray*}
 \Omega^+_s= \{ x | x \notin \Omega, \mathrm{dist}(x,\partial \Omega) < \frac{\delta}{2} \}.
\end{eqnarray*}
Applying the $L^p$ estimate to Eq.(\ref{G_transd}) in $\Omega_s$ we obtain
\begin{eqnarray}
\| G \|_{W^{2-1/p,p}(\partial \Omega)} & \le & C \| G \|_{W^{2,p}(\Omega_s)} \le C(\mathbf{F}) \| G \|_{L^p(\Omega_s^{'})}
\le C(\mathbf{F}) \| G \|_{L^{\infty}(\Omega_s^{'})} \nonumber \\
& := & C_G \| G \|_{L^{\infty}(\Omega_s^{'})} , \label{trace_est}
\end{eqnarray}
where the second inequality is a consequence of the $L^p$ estimate of the solution
to $-\nabla \cdot ( \epsilon \mathbf{F} \nabla G ) = 0$ in $\Omega'_s$. The coefficient $C_G=C(\mathbf{F})$
depends on the ellipticity constants of $\mathbf{F}$ and its moduli of continuity on $\Omega_s$, hence
is bounded as long as $\mathbf{F}$ is bounded. By combining Eqs.(\ref{trace_est}) and (\ref{g_est}) we
get (c). For the last estimate we notice
\begin{eqnarray}
\|\mathbf{F} \nabla G \|_{W^{1-1/p,p}(\Gamma)} & \le &
C \| [\epsilon] \mathbf{F} \nabla G \|_{W^{1,p}(\Omega_s)}  \nonumber \\
& \le & C  \| \mathbf{F} \|_{W^{1,p}(\Omega_s)} \| \nabla G \|_{W^{1,p}(\Omega_s)}, ~~~ (p > 3) \nonumber \\
& \le & C  \| \mathbf{F} \|_{W^{1,p}(\Omega_s)}  \| G \|_{W^{2,p}(\Omega_s)} \nonumber \\
& \le & C(\mathbf{F})  \| \mathbf{F} \|_{W^{1,p}(\Omega_s)}  \| G \|_{L^p(\Omega'_s)} \nonumber \\
& \le & C  \| \mathbf{F} \|_{W^{1,p}(\Omega_s)}  \| G\|_{L^{\infty}(\Omega'_s)}  \nonumber \\
& := & C_{\Gamma} \| G\|_{L^{\infty}(\Omega'_s)} . \qed \label{fG_estimate}
\end{eqnarray}
\end{proof}
\begin{remark}
$\| G \|_{L^{\infty}(\Omega'_s)}$ can also be estimated by Eq.(\ref{green_estim}) if $\delta$ is replaced by
$\delta/2$ and $\| J\|_{L^{\infty}(\Omega)}$ is replaced by $\| J\|_{L^{\infty}(\Omega \cup \Omega^+)}$.
\end{remark}

\subsection{Regularity and estimates for the regular linearized solution component $\phi^r$}
We consider an elliptic interface problem modified from the Poisson-Boltzmann equation
\begin{equation} \label{reg_LPBeq}
\begin{split}
-\nabla \cdot (\epsilon \mathbf{F} \nabla \phi^r ) + J \kappa^2 (\phi^r + G)  = & ~
\nabla \cdot ((\epsilon - \epsilon_m) \mathbf{F} \nabla G) + f ~\mbox{in}~ \Omega \\
\left [ \phi^r \right ] = \phi^r_s - \phi^r_m  = & ~ 0, ~\mbox{on}~ \Gamma  \\
\phi^r  = & ~g - G ~\mbox{on}~ \partial \Omega,
\end{split}
\end{equation}
where $[\epsilon] = \epsilon_s - \epsilon_m$ is the jump of
dielectric constant and $f \in L^{p}(\Omega)$ is a given function.
The equation for the regular potential solution of the linear
Poisson-Boltzmann equation is a special case of (\ref{reg_LPBeq})
with $f=0$. We remark that the regular component of the linear
Poisson-Boltzmann equation in the absence of the Piola
transformation represents a typical elliptic equation with
discontinuous coefficients, for which Theorem (\ref{w2p_theorem})
can be directly applied to get the existence and the estimate. In
fact, the potential solution in this case is smooth in every
subdomain (Proposition 1.4, \cite{Yanyan_CompositeMaterial}). When
the Piola transformation is incorporated, the coefficients of the
Eq. (\ref{reg_LPBeq}) are not smooth and we have to rebuild
the regularity and the estimate of the regular potential solution
$\phi^r$, as summarized in the following theorem
\begin{theorem} \label{reg_LPBeq_theorem}
There exists a unique solution $\phi^r$ of (\ref{reg_LPBeq}) in $H^1(\Omega)$. Moreover, there exists a positive constant $C_f$ such that if $\| \mathbf{F} - \mathrm{I} \|_{W^{1,p}(\Omega)} \le C_f$ then $\phi^r$ belongs to
$\mathcal{W}^{2,p}(\Omega)$ and the following estimate holds true
\begin{eqnarray}
\| \phi^r \|_{\mathcal{W}^{2,p}(\Omega)}  & \le &  C_2 \left ( \| G \|_{L^p(\Omega_s)} + \|f \|_{L^p(\Omega)} +
\| g-G \|_{W^{2-1/p,p}(\partial \Omega)} +  \nonumber \right . \\
 & & \qquad \left . \|\mathbf{F} \nabla G \|_{W^{1-1/p,p}(\Gamma)} \right ).
\label{pr_lpbe_estimate}
\end{eqnarray}
\end{theorem}
Before proving this $\mathcal{W}^{2,p}$ estimate, we first establish a lemma concerning the $L^{\infty}$
estimate of a linear elliptic interface problem.
\begin{lemma} \label{linfty_lemma}
Let $\phi^r$ solve
\begin{eqnarray*}
-\nabla \cdot (\epsilon \nabla \phi^r ) + b \phi^r  & = & f ~\mbox{in}~ \Omega  \\
\left [ \phi^r \right ]  & = &  0, ~\mbox{on}~ \Gamma  \\
\left [ \epsilon \phi^r_n \right ]  & = &  0, ~\mbox{on}~ \Gamma  \\
\phi^r   & = & g  ~\mbox{on}~ \partial \Omega,
\end{eqnarray*}
where $\epsilon$ is a piecewise constant as defined for problem (\ref{reg_LPBeq}) and $b>0$ is a given real number,
$f(x) \in L^{p}(\Omega), g \in H^1(\Omega), p >3$. Then
\begin{eqnarray}
\| \phi^r \|_{L^{\infty}} \le  C \left( \| f \|_{L^{p}(\Omega)} + \| g \|_{H^{1/2}(\partial \Omega)} \right) \label{linfty_estimate}
\end{eqnarray}
\end{lemma}
\begin{proof}
The existence of unique solution $\phi^r \in \mathcal{W}^{2,p}
\subset H^1(\Omega)$ can be directly deduced from Theorem
(\ref{w2p_theorem}). We follow \cite{Chen_rpbe} and let $\phi^r =
\phi^l + \phi^n$ where $\phi^l$ solves
\begin{eqnarray*}
-\nabla \cdot (\epsilon \nabla \phi^l )  & = & f ~\mbox{in}~ \Omega  \\
\left [ \phi^l \right ]  & = &  0 ~\mbox{on}~ \Gamma,  \\
\left [ \epsilon \phi^l_n \right ]  & = &  0 ~\mbox{on}~ \Gamma,  \\
\phi^l   & = & g  ~\mbox{on}~ \partial \Omega,
\end{eqnarray*}
and $\phi^n$ solves
\begin{eqnarray*}
-\nabla \cdot (\epsilon \nabla \phi^n ) + b (\phi^n + \phi^l) & = & 0 ~\mbox{in}~ \Omega,  \\
\left [ \phi^n \right ]  & = &  0 ~\mbox{on}~ \Gamma,  \\
\left [ \epsilon \phi^n_n \right ]  & = & 0 ~\mbox{on}~ \Gamma,  \\
\phi^n   & = & 0  ~\mbox{on}~ \partial \Omega.
\end{eqnarray*}
It is well known \cite{Babuska_70,ChenZM_98} that
\begin{eqnarray*}
\| \phi^l \|_{L^{\infty}} \le C \left( \| f\|_{L^p(\Omega)} + \| g\|_{H^{1/2}(\partial \Omega)} \right ),
\end{eqnarray*}
while for $\phi^n$ we claim that $ - \| \phi^l \|_{L^{\infty}} \le
\| \phi^n \|_{L^{\infty}} \le \| \phi^l \|_{L^{\infty}}$. To prove
this assertion we define $\phi_t = \max(\phi^n - \alpha,0)$ where
$\alpha = \| \phi^l \|_{L^{\infty}}$. Then the trace $\trace(\phi_t) =0$ hence
$\phi_t \in H^1_0(\Omega)$ by definition. Consider the weak
formulation of the problem for $\phi^n$ with test function $\phi_t$
\begin{eqnarray*}
(\epsilon \nabla \phi^n, \nabla \phi_t) + b(\phi^n+\phi^l,\phi_t) =
0.
\end{eqnarray*}
Since $\phi_t \ge 0$ wherever $\phi^n \ge \alpha$, we have
\begin{eqnarray*}
b(\phi^n+\phi^l,\phi_t) = \int_{\phi^n \ge \alpha} b(\phi^n+\phi^l)
\phi_t dx + \int_{\phi^n < \alpha} b(\phi^n+\phi^l) \phi_t dx \ge 0,
\end{eqnarray*}
and
\begin{eqnarray*}
0 \ge (\epsilon \nabla \phi^n, \nabla \phi_t) = (\epsilon \nabla (\phi^n-\alpha), \nabla \phi_t) =
(\epsilon \nabla \phi_t, \nabla \phi_t) \ge 0.
\end{eqnarray*}
Thus $\nabla \phi_t =0$, and $\phi_t =0$ or $\phi^n \le \alpha$ in $\Omega$ follows from the
Poincare inequality. By defining $\phi_t = \min(\phi^n + \alpha,0)$ and following the
same procedure we can verify that $\phi^n \ge -\alpha$. The lemma shall be finally proved
by combining the estimates of $\phi^l$ and $\phi^n$.
\qed
\end{proof}

\begin{proof}{\bf of Theorem (\ref{reg_LPBeq_theorem}).}
Consider the general weak formulation of the elliptic equation in problem (\ref{reg_LPBeq}), i.e., find
$\phi^r = u \in H^1_0(\Omega)$ such that $A(u,v) = F(v), \forall v \in H^1_0(\Omega)$ where
\begin{eqnarray*}
A(u,v) &=& \int_{\Omega} ( \mathbf{F} \nabla u \nabla v + J \kappa^2 u v) dx, \\
F(v) &=& \int_{\Omega} \left ( \nabla \cdot ((\epsilon - \epsilon_m) \mathbf{F} \nabla G) - J \kappa^2 G + f \right )  dx - A(g-G,v).
\end{eqnarray*}
We shall apply the Lax-Milgram theorem to obtain the existence and
the uniqueness of a weak solution $\phi^r \in H^1(\Omega)$ to
(\ref{reg_LPBeq}). Hence we must show that $F(\cdot)$ is bounded, and
$A(\cdot,\cdot)$ is bounded and coercive with the assumptions on the
coefficient matrix $\mathbf{F}$ and the Jacobian $J$. Consider the
bilinear form $A(\cdot,\cdot)$. The Piola transform matrix
$\mathbf{F}$ is positive definite, hence $\mathbf{F} \nabla v \cdot
\nabla v \ge \gamma | \nabla v |^2$ for some $\gamma >0$. This inequality 
and the positiveness of Jacobian $J$ give
\begin{eqnarray}
A(v,v) & = & \int_{\Omega} (\mathbf{F} \nabla v \cdot \nabla v + J \kappa^2 v^2 ) dx \ge
\int_{\Omega} (\gamma | \nabla v|^2 + J \kappa^2 v^2 )dx  \ge \lambda |u|^2_{H^1(\Omega)} \nonumber \\
& = & \gamma \left(  \frac{1}{2} |v|^2_{H^1(\Omega)} + \frac{1}{2} |v|^2_{H^1(\Omega)} \right) \ge
 \gamma \left (  \frac{1}{2 \theta^2} \| v \|^2_{L^2(\Omega)} + \frac{1}{2} |v|^2_{H^1(\Omega)} \right ) \nonumber \\
& \ge & m \left (  \| v \|^2_{L^2(\Omega)} + |v |^2_{H^1(\Omega)} \right ) = m \| v\|^2_{H^1(\Omega)},
\end{eqnarray}
where in the second inequality we applied the Poincare inequality
with constant $\theta$. Thus we verified that $A(\cdot,\cdot)$ is
coercive, with coercivity constant $m=\min\{ \gamma/(2 \theta^2),
\gamma/2 \}$.

On the other hand,
\begin{eqnarray}
|A(u,v)| & = & | \int_{\Omega} ( \mathbf{F} \nabla u \cdot \nabla v| + J \kappa^2 u v ) dx | \\
& \le & \sum_{i,j} \int_{\Omega} | \mathbf{F}_{ij} D_i u D_j v | dx  + \int_{\Omega} |J \kappa^2 uv  | dx \nonumber \\
& \le & \sum_{i,j} \| \mathbf{F}_{ij} \|_{L^{\infty}(\Omega)}  \| D_i u D_j v \|_{L^1(\Omega)} +
\kappa^2 \| J \|_{L^{\infty}} \| uv \|_{L^1(\Omega)} \nonumber \\
& \le &  \sum_{i,j} \| \mathbf{F}_{ij} \|_{L^{\infty}} \| u\|_{H^1(\Omega)} \|v \|_{H^1(\Omega)} + \kappa^2 \| J\|_{L^{\infty}}
\| u \|_{L^2(\Omega)} \| v \|_{L^2{\Omega}}  \nonumber \\
& \le & K_1 \| u \|_{H^1(\Omega)} \| v\|_{H^1(\Omega)}
\end{eqnarray}
which proves that $A(\cdot,\cdot)$ is bounded with constant
$\displaystyle{ K_1 = \sum_{i,j} \| \mathbf{F}_{ij} \|_{L^{\infty}} +
\kappa^2 \| J \|_{L^{\infty}} }$. This constant $K_1$ is finite because $\mathbf{F},J$ belong to $W^{1,p}(\Omega)$ which is
compactly embedded in $C^0(\overline{\Omega})$ for $p >3$.

In order to apply the Lax-Milgram theorem it remains to show that $F(v)$ is bounded on $H^1_0(\Omega)$. We have
\begin{eqnarray*}
| F(v) | & \le & | \int_{\Omega} (\epsilon - \epsilon_m) \mathbf{F} \nabla G v dx | +
\int_{\Omega} | J \kappa^2 Gv  + fv | dx + | A(g-G,v) | \nonumber   \\
& = & | \int_{\Omega_m} (\epsilon_m - \epsilon_m) \mathbf{F} \nabla G v dx  +
   \int_{\Omega_s} (\epsilon_s - \epsilon_m) \mathbf{F} \nabla G v dx | + \\
 & & ~  \int_{\Omega} | J \kappa^2 G v + fv| dx + | A(g-G,v) | \nonumber   \\
& = & | \int_{\Omega_s} (\epsilon_s - \epsilon_m) \mathbf{F} \nabla G  v dx | +
\int_{\Omega} | J \kappa^2 G v + fv | dx + | A(g-G,v) | \nonumber  \\
& \le &  \int_{\Omega_s} | (\epsilon_s - \epsilon_m) \mathbf{F} \nabla G  v| dx  +
\int_{\Omega} | J \kappa^2 G v + fv | dx + | A(g-G,v) | \nonumber  \\
& \le & [\epsilon] \| \mathbf{F} \nabla G \|_{L^2(\Omega)} \| v \|_{L^2(\Omega)}
+ (\kappa^2 \| J G\|_{L^2(\Omega)}  + \|f \|_{L^2(\Omega)}) \|v \|_{L^2(\Omega)} + \\
& & K_1 \| g-G\|_{H^1(\Omega)} \| v\|_{H^1(\Omega)} \nonumber \\
& = & \left ( [\epsilon] \| \mathbf{F} \nabla G \|_{L^2(\Omega)}
+ \kappa^2 \| J G\|_{L^2(\Omega)}  + \|f \|_{L^2(\Omega)}
\right.
\\
& & ~~~~~~ \left.
 + K_1 \| g-G\|_{H^1(\Omega)}  \right ) \| v\|_{H^1(\Omega)}  \nonumber \\
& = & K_2 \| v\|_{H^1(\Omega)},
\end{eqnarray*}
hence $F(\cdot)$ is a bounded linear functional on $H^1_0(\Omega)$.

We now proceed to show the regularity result and the estimate of $\phi^r$ following the
similar iterative technique in \cite{Grandmont_3Dcouple}. For this purpose we introduce a sequence
$\{ \phi^r_N \}$ generated by
\begin{eqnarray}
-\nabla \cdot (\epsilon \nabla \phi^r_N ) + J \kappa^2 \phi^r_N  & = &
\nabla \cdot ((\epsilon - \epsilon_m) \mathbf{F} \nabla G) - J \kappa^2 G \\
& & ~~~~~~ +
\nabla \cdot (\epsilon (\mathbf{F}-\mathrm{I}) \nabla \phi^r_{N-1} ) ~\mbox{in}~ \Omega, \nonumber \\
\left [ \phi^r_N \right ] & = &  0 ~\mbox{on}~ \Gamma,  \label{reg_LPBeq_n}\\
\phi^r_N   & = & g - G ~\mbox{on}~ \partial \Omega, \nonumber
\end{eqnarray}
and prove that $\phi^r_N \in \mathcal{W}^{2,p}(\Omega)$ and $\phi^r_N$ converges to the unique solution $\phi^r$ of
(\ref{reg_LPBeq}) in $\mathcal{W}^{2,p}(\Omega)$ as $N \rightarrow \infty$. The first term $\phi^r_0$ of the sequence solves
\begin{eqnarray}
-\nabla \cdot (\epsilon \nabla \phi^r_0 ) + J \kappa^2 \phi^r_0  & = &
\nabla \cdot ((\epsilon - \epsilon_m) \mathbf{F} \nabla G) - J \kappa^2 G ~\mbox{in}~ \Omega \nonumber \\
\left [ \phi^r_0 \right ] & = &  0, ~\mbox{on}~ \Gamma  \label{reg_LPBeq_0}\\
\phi^r_0   & = & g - G ~\mbox{on}~ \partial \Omega, \nonumber
\end{eqnarray}
therefore it belongs to $\mathcal{W}^{2,p}(\Omega)$ according to Theorem (\ref{w2p_theorem}). Suppose now
that $\phi^r_{N-1} \in \mathcal{W}^{2,p}(\Omega)$, then
$\nabla \phi^r_{N-1} \in \mathcal{W}^{1,p}(\Omega)$ and
$\nabla \cdot (\epsilon (\mathbf{F} - \mathrm{I}) \nabla \phi^r_{N-1}) \in L^p(\Omega)$
following from Lemma (\ref{w1p_algebra}). Thus problem (\ref{reg_LPBeq_n}) also has a unique solution $\phi^r_N
\in \mathcal{W}^{2,p}(\Omega)$ for all integer $N$ according to Theorem (\ref{w2p_theorem}). To prove that
$\phi^r_N$ converges to the unique solution $\phi^r$ of problem (\ref{reg_LPBeq}), we estimate
$\| \phi^r_N - \phi^r_{N-1} \|_{\mathcal{W}^{2,p}(\Omega)}$ and show it is decreasing as $N \rightarrow \infty$.
By subtracting the equations in (\ref{reg_LPBeq}) for $N$ from those for $N-1$ we obtain a problem for
$\phi^r_N - \phi^r_{N-1}$. Applying Theorem (\ref{w2p_theorem}) again we know that this problem has a
unique solution in $\mathcal{W}^{2,p}(\Omega)$ which has an estimate
\begin{eqnarray}
\| \phi^r_N - \phi^r_{N-1} \|_{\mathcal{W}^{2,p}(\Omega)} & \le &
C \left( \| \nabla \cdot (\epsilon (\mathbf{F} - \mathrm{I}) \nabla (\phi^r_{N-1} - \phi^r_{N-2}) ) \|_{L^p(\Omega)}
\right.  \nonumber \\
& & ~~~~~ \left.
+ \| \phi^r_N - \phi^r_{N-1} \|_{L^p(\Omega)} \right) \nonumber \\
& \le & C \| \nabla \cdot (\epsilon (\mathbf{F} - \mathrm{I}) \nabla (\phi^r_{N-1} - \phi^r_{N-2}) ) \|_{L^p(\Omega)},
\nonumber \\
& \le & C \| \mathbf{F} - \mathrm{I} \|_{W^{1,p}(\Omega)}
\| \phi^r_{N-1} - \phi^r_{N-2} \|_{\mathcal{W}^{2,p}(\Omega)},
 \label{w2p_prndiff_1}
\end{eqnarray}
where in the second inequality we applied Lemma (\ref{linfty_lemma}) to the problem for $(\phi^r_N - \phi^r_{N-1})$,
and the generic constant $C$ is independent of $N, \mathbf{F}$. Therefore if the constant $C_f$ in the assumption of the theorem is chosen such that $C C_f = k < 1$ then $\| \phi^r_N - \phi^r_{N-1} \|_{\mathcal{W}^{2,p}(\Omega)}$
is decreasing with respect to $N$ hence the sequence $\phi^r_n$ converges to a unique
element $\overline{\phi^r}$ in $\mathcal{W}^{2,p}(\Omega)$. Letting $N \rightarrow \infty$ we can observe
that $\overline{\phi^r}$ is the unique solution of problem (\ref{reg_LPBeq}), meaning
$\overline{\phi^r} = \phi^r$.

The estimate of $\phi^r$ is obtained by estimating $\phi^r_N$ and passing $N$ to $\infty$. We notice
that $\phi^r_N = \phi^r_N - \phi^r_{N-1} + \phi^r_{N-1} - \phi^r_{N-2} + \cdots + \phi^r_0$, hence
\begin{eqnarray*}
\|\phi^r_N \|_{\mathcal{W}^{2,p}} & \le & \frac{1-k^{N-1}}{1-k} \| \phi^r_1 - \phi^r_0 \|_{\mathcal{W}^{2,p}}
+ \| \phi^r_0 \|_{\mathcal{W}^{2,p}} \nonumber \\
& \le & C_2 (\| G \|_{L^p(\Omega_s)} + \| g-G \|_{W^{2-1/p,p}(\Omega)} 
\\
& & ~~~~~ + \|\mathbf{F} \nabla G \|_{W^{1-1/p,p}(\Gamma)}~\mbox{as}~N \rightarrow \infty,
\end{eqnarray*}
where both Theorem (\ref{w2p_theorem}) and Lemma (\ref{linfty_lemma}) are applied to the problem of
$ \phi^r_1 - \phi^r_0$ and the problem of $\phi^r_0$ to get the desired bounds with respect to
the $\mathcal{W}^{2,p}$ and $L^{\infty}$ norms, and $C_2$ absorbs $k$ and
all the generic constants involved in these bounds.
\qed
\end{proof}

\subsection{Regularity and estimates for the regular nonlinear solution component $\phi^r$}

For the nonlinear Poisson-Boltzmann equation, the regular component $\phi^r$ of its potential solution solves
\begin{equation} \label{reg_NLPBeq}
\begin{split}
-\nabla \cdot (\epsilon \mathbf{F} \nabla \phi^r ) + J \kappa^2 \sinh(\phi^r + G)   = & ~
\nabla \cdot ((\epsilon - \epsilon_m) \mathbf{F} \nabla G) ~\mbox{in}~ \Omega  \\
\left [ \phi^r \right ] = \phi^r_s - \phi^r_m  = &~  0 ~\mbox{on}~ \Gamma  \\
\phi^r   = &~ g - G ~\mbox{on}~ \partial \Omega.
\end{split}
\end{equation}
The appearance of the nonlinear function $\sinh(x)$ complicates the
establishment of the existence of $\phi^r$. In particular, the
Lax-Milgram Theorem is not applicable to problem (\ref{reg_NLPBeq}).
Instead we define a energy functional based on the weak formulation
of (\ref{reg_NLPBeq}) and show that the unique minimizer of this
energy functional is the unique solution of (\ref{reg_NLPBeq}). On
the other hand, the establishment of the regularity and
$\mathcal{W}^{2,p}$ estimate of $\phi^r$ for (\ref{reg_NLPBeq}) is
simplified thanks to Theorem (\ref{reg_LPBeq_theorem}).

We start with the weak formulation of (\ref{reg_NLPBeq}):
$$ \mbox{Find}~\phi^r \in M \equiv \{ v \in H^1(\Omega) | e^v,e^{-v} \in L^2(\Omega),
\mbox{and}~v=g-G~ \mbox{on} ~ \partial \Omega \},$$
such that
\begin{eqnarray}
A(\phi^r,v) + (B(\phi^r),v) + \langle f_G,v \rangle =0, ~\forall v
\in H^1_o(\Omega), \label{pr_eq_weak}
\end{eqnarray}
where
\begin{eqnarray*}
A(\phi^r,v) & = & (\epsilon \mathbf{F} \nabla \phi^r, \nabla v), ~
(B(\phi^r),v) = (J \kappa^2 \sinh(\phi^r + G),v), \label{AB_def} \\
\langle f_G,v \rangle & = & \int_{\Omega} (\epsilon - \epsilon_m) \mathbf{F} \nabla G \cdot \nabla v. \label{fG_def}
\end{eqnarray*}
We also use $f_G$ to denote the function $[\epsilon] \mathbf{F}(\mathbf{u}) \nabla G \cdot \mathbf{n}$ on the
dielectric boundary $\Gamma$, since
\begin{eqnarray}
\langle f_G,v \rangle & = & ([\epsilon] \mathbf{F} \nabla G \cdot \mathbf{n}, v), \label{fG_def2}
\end{eqnarray}
where $[\epsilon] = \epsilon_s - \epsilon_m$ is the jump in $\epsilon$ on $\Gamma$.
Based on this weak formulation we define an energy on $M$:
\begin{eqnarray}
 E(w) = \int_{\Omega} \frac{\epsilon}{2} \mathbf{F} \nabla w \cdot \nabla w + J \kappa^2 \cosh(w + G)
+ \langle f_G,w \rangle. \label{eng_functional}
\end{eqnarray}
The weak solution of Eq.(\ref{pr_transd}) can be characterized as the minimizer of this energy functional. This
equivalence and the existence of this minimizer are due to the following four
simple lemmas. For the proof of these lemmas we refer to \cite{Chen_rpbe}.
\begin{lemma}
If $u$ is the solution of the optimization problem, i.e.,
$$ E(u) = \inf_{w \in M} E(w),$$
then $u$ is the solution of (\ref{pr_transd}).
\end{lemma}
\begin{lemma}
Let $F(u)$ be a functional defined on $M$, if
\begin{enumerate}
\item $M$ is weakly sequential compact, and
\item $F$ is weakly lower semi-continuous on $M$,
\end{enumerate}
then there exists $u \in M$ such that
$$ F(u) = \inf_{w \in M} F(w). $$
\end{lemma}
\begin{lemma}
The following results hold true
\begin{enumerate}
\item Let $V$ be a reflective Banach space. The set $M := \{ v \in V | \| v \| \le r_0 \}$ is weakly
sequential compact.
\item if $\displaystyle{ \lim_{\| v \| \rightarrow \infty} F(v) = \infty }$, then
$\displaystyle{ \inf_{w \in V} F(w) = \inf_{w \in M} F(w) }$
\end{enumerate}
\end{lemma}
\begin{lemma}
If $F$ is a convex functional on a convex set $M$ and $F$ is G\^ateaux differentiable, then $F$ is w.l.s.c. on $M$.
\end{lemma}
The existence and the uniqueness of the weak solution to (\ref{pr_transd}) can be established using these lemmas. The
following lemma establishes the existence of the minimizer of the energy $E(w)$.
\begin{theorem} \label{E_minimizer}
There exists a unique $u \in M \subset H^1 (\Omega)$ such that
$$ E(u) = \inf_{w \in M} E(w).$$
\end{theorem}
\begin{proof}
The differentiability of $E(w)$ follows its definition. Actually we have
$$ \langle DE(u), v \rangle = A(u,v) + (B(u),v) + \langle f_G,v \rangle.$$
The minimizer of $E(w)$ exists if we can prove that
\begin{enumerate}
\item $M$ is a convex set
\item $E$ is convex on $M$
\item $\displaystyle{ \lim _{\|w \|_{H^1(\Omega)} \rightarrow \infty } E(w) = \infty }$
\end{enumerate}
It is easy to verify (1). The convexity of $\mathbf{F} \nabla w \cdot \nabla w $ follows from
the fact that
$$  0 \le  \mathbf{F}\nabla (\gamma w) \cdot \nabla (\gamma w) =
   \gamma^2 \mathbf{F} \nabla w \cdot \nabla w \le
   \gamma \mathbf{F} \nabla w \cdot  \nabla w $$
for any $0 \le \gamma \le 1$ since $\mathbf{F}$ is positive definite. The convexity of $\cosh(w+G)$
follows from the convexity of $\cosh(x)$ directly. Actually $E(w)$ is strictly convex. To prove (3) we only need
to show that
\begin{eqnarray}
 E(w) \ge C(\epsilon,\kappa,\mathbf{F}) \| w \|^2_{H^1(\Omega)} + C(G,g). \label{Mini_cond}
\end{eqnarray}
We notice that $\cosh(x) \ge 1$ and
\begin{eqnarray*}
<f_G,w> & \le & \epsilon_s \| \mathbf{F} \nabla G\|_{L^2(\Omega_s)} \| \nabla w\|_{L^2(\Omega_s)} \\
          & \le & \epsilon_s \| \mathbf{F} \|_{L^2(\Omega_s)} \| \nabla G\|_{L^2(\Omega_s)} \| \nabla w\|_{L^2(\Omega_s)} \\
          & \le & \frac{\epsilon_s \gamma}{2} (\| \nabla G\|^2_{L^2(\Omega_s)} + \| \nabla w\|^2_{L^2(\Omega_s)} ),
\end{eqnarray*}
where the matrix norm
$$ \gamma =
\|\mathbf{F}\|_{L^2(\Omega_s)} = \sup_{v \in M, \|v\|_{L^2(\Omega)} = 1} \|  \mathbf{F} v\|_{L^2(\Omega_s)}, $$
is finite because $\mathbf{F}$ is continuous. Therefore
\begin{eqnarray*}
E(w) & \ge & \frac{1}{2} \int_{\Omega} \epsilon \mathbf{F} \nabla w \cdot \nabla w - |<f_G,w>| \\
     & \ge & \frac{\gamma}{2} \big( \int_{\Omega_s} \epsilon_s |\nabla w|^2 + \int_{\Omega_m} \epsilon_m |\nabla w|^2
     \big )
      - \frac{\epsilon_s \gamma}{2} (\| \nabla G\|^2_{\Omega_s} + \| \nabla w\|^2_{L^2(\Omega_s)} ) \\
     & = & \frac{\gamma}{2}  \int_{\Omega_m} \epsilon_m |\nabla w|^2 -
           \frac{\epsilon_s \gamma}{2} \| \nabla G\|^2_{L^2(\Omega_s)} \\
     & \ge & C(\epsilon, \gamma) \| \nabla w\|_{L^2(\Omega)}^2 - \frac{\epsilon_s \gamma}{2} \| \nabla G\|^2_{L^2(\Omega_s)}.
\end{eqnarray*}
The inequality (\ref{Mini_cond}) follows from the equivalence of $\|\nabla w\|_{L^2(\Omega)}$ and
$\|w\|_{H^1(\Omega)}$ on set $M$. The uniqueness of the minimizer of $E(w)$ comes from the strict convexity
of $E$.  \qed
\end{proof}

\begin{theorem} \label{reg_NLPBeq_theorem}
There exists a unique solution $\phi^r$ of (\ref{reg_LPBeq}) in $H^1(\Omega)$. Moreover, there
exist constants $C_1,C_2$ and $C_3$ such that $\phi^r$ is bounded by
\begin{eqnarray}
\| \phi^r \|_{L^{\infty}(\Omega)}
\le C_1 + C_2 \| J\|_{L^{\infty}(\Omega)}  + C_3 \| J\|_{L^{\infty}(\Omega)}
\| \mathbf{F} \|_{W^{1,p}(\Omega_s)}. \label{linfty_nlpbe_estimate}
\end{eqnarray}
\end{theorem}
\begin{proof}
The existence of the solution $\phi^r$ in $H^1(\Omega)$ has been proved by theorem (\ref{E_minimizer}) and its
four lemmas. It remains to verify the $L^{\infty}$ bounds of $\phi^r$. Let $\phi^r = \phi^l + \phi^n$ be decomposed
into a linear component $\phi^l$ and a nonlinear component.
The linear component $\phi^l$ satisfies Eq.(\ref{pl_transd}). The existence of a weak
solution $\phi^l \in H^1_0(\Omega)$ follows that $ \nabla \cdot ((\epsilon - \epsilon_m) \mathbf{F}\nabla G)$ is an operator in $H^{-1}(\Omega)$ \cite{Weinberger_greenfunc}. It is well known that in general
\begin{eqnarray}
C_1 \| \phi^l\|_{L_{\infty}} \le \| \phi^l \|_{H^1} \le C_2 ( \|g-G\|_{H^{1/2}(\partial \Omega)} + \|f_G\|_{H^{1/2}(\Gamma)} ). \label{pl_estimate}
\end{eqnarray}

To estimate the nonlinear component we follow \cite{Chen_rpbe} and define:
\begin{eqnarray*}
\alpha ' & = &
   \arg \max_{c} \{ J \kappa^2 \sinh(c + \sup_{x \in \Omega_s} \phi^l  + \sup_{x \in \Omega_s} G ) \le 0 \}, \\
\beta ' & = &
   \arg \min_{c} \{ J\kappa^2 \sinh(c + \inf_{x \in \Omega_s} \phi^l  + \inf_{x \in \Omega_s} G ) \ge 0 \}, \\
\alpha & = & \min (\alpha',0), \\
\beta  & = & \max (\beta', 0).
\end{eqnarray*}
It follows from the monotonicity of $\sinh(x)$ that
\begin{eqnarray}
\beta  & = & \| \phi^l \|_{L^{\infty}(\Omega_s)} + \| G \|_{L^{\infty}(\Omega_s)}, \label{beta_final} \\
\alpha & = & -\beta. \label{alpha_final}
\end{eqnarray}
We will show that $\alpha$ and $\beta$ are the lower and upper $L^{\infty}$ bounds of the nonlinear
component $\phi^n$ of the weak solution to (\ref{reg_NLPBeq}), following the similar
procedure as that used in proving Lemma (\ref{linfty_lemma}).

Define $$ \phi_t = \max (\phi^n - \beta, 0),$$
then $ \trace \phi_t =0 $ since $\phi^n \in H^1_0(\Omega)$ and $\beta >0$ by definition.
Therefore $\phi_t \in H_0^1(\Omega)$ and satisfies the weak formulation of Eq.(\ref{pn_transd}):
$$ (\epsilon \mathbf{F} \nabla \phi^n, \nabla \phi_t) + \big( J \kappa^2
\sinh(\phi^n + \phi^l + G), \phi_t \big) = 0.$$
Since $\phi_t \ge 0$ wherever $\phi^n \ge \beta$, we have
$$ J(\mathbf{u}) \kappa^2 \sinh(\phi^n + \phi^l + G) \ge
   J(\mathbf{u}) \kappa^2 \sinh(\beta + \inf_{x \in \Omega_s} \phi^l + \inf_{x \in \Omega_s} G ) \ge 0.$$
Therefore
$$ 0 \ge (\epsilon \mathbf{F} \nabla \phi^n, \nabla \phi_t) =
          (\epsilon \mathbf{F} \nabla (\phi^n - \beta), \nabla \phi_t)
          = \epsilon \mathbf{F} \nabla \phi_t \cdot \nabla \phi_t \ge 0,$$
where the last inequality holds true since $\mathbf{F}$ is positive definite.
Hence $\nabla \phi_t=0$, and $\phi_t=0$ or $\phi^n \le \beta$ in $\Omega$ follows from the Poincare
inequality. This establishes the upper bound of $\phi^n$. By changing $\phi_t$ to be $\min(\phi^n + \alpha, 0)$
we can also prove that $\alpha$ is the lower bound.

Combining the estimates for $\phi^l$ and $\phi^n$ we finally obtain the $L^{\infty}$ estimate
of the regular component $\phi^r$:
\begin{eqnarray*}
\| \phi^l + \phi^n \|_{L^{\infty}(\Omega)} & \le &
\|\phi^l \|_{L^{\infty}(\Omega)} + \| \phi^n \|_{L^{\infty}(\Omega)} 
\\
& \le &
\|\phi^l \|_{L^{\infty}(\Omega)} + \| \phi^l \|_{L^{\infty}(\Omega_s)}  + \| G \|_{L^{\infty}(\Omega_s)} \nonumber \\
& \le & 2(C_g + C_G \| J\|_{L^{\infty}(\Omega)} + C_{f_G} \| \mathbf{F} \|_{W^{1,p}(\Omega_s)} \| J\|_{L^{\infty}(\Omega)} 
\\
& & ~~~~~ +
\frac{\| J\|_{L^{\infty}(\Omega)} N K q_{max}}{\delta}  \nonumber \\
& = & C_1 + C_2 \| J\|_{L^{\infty}(\Omega)}  + C_3 \| J\|_{L^{\infty}(\Omega)}
\| \mathbf{F} \|_{W^{1,p}(\Omega_s)}. \qed \label{phis_estimate}
\end{eqnarray*}
\end{proof}

We are now able to examine the regularity results and the estimate of $\phi^r$ in $\mathcal{W}^{2,p}$.
\begin{theorem}
If $\| \mathbf{F} - \mathrm{I} \|_{W^{1,p}(\Omega)} \le C_f$ then the unique solution $\phi^r$
of (\ref{reg_NLPBeq}) belongs to $\mathcal{W}^{2,p}(\Omega)$ and the following estimate holds
\begin{eqnarray}
\hspace*{-0.5cm}
\| \phi^r \|_{\mathcal{W}^{2,p}(\Omega)} \le C \left ( \| G \|_{L^p(\Omega_s)} +
\| g -G \|_{W^{2-1/p,p}(\partial \Omega)}
+ \|\mathbf{F} \nabla G \|_{W^{1-1/p,p}(\Gamma)} \right ) . \label{pr_nlpbe_estimate}
\end{eqnarray}
\end{theorem}
\begin{proof}
It is noticed that the problem (\ref{reg_NLPBeq}) can be written as a form similar to its linear counterpart
(\ref{reg_LPBeq})
\begin{eqnarray*}
-\nabla \cdot (\epsilon \mathbf{F} \nabla \phi^r ) + J \kappa^2 (\phi^r + G)  & = &
\nabla \cdot ((\epsilon - \epsilon_m) \mathbf{F} \nabla G) \\ 
& & - J \kappa^2 ( \sinh(\phi^r + G) - (\phi^r + G)) ~\mbox{in}~ \Omega  \\
\left [ \phi^r \right ] & = &  0, ~\mbox{on}~ \Gamma  \\
\phi^r   & = & g - G ~\mbox{on}~ \partial \Omega.
\end{eqnarray*}
According to Theorem (\ref{reg_LPBeq_theorem}), the $\mathcal{W}^{2,p}$ regularity of $\phi^r$
directly follows from the facts that
$\nabla \cdot ((\epsilon - \epsilon_m) \mathbf{F} \nabla G)$ represents an interface
condition in $W^{2-1/p,p}(\Omega)$ and that $J \kappa^2 ( \sinh(\phi^r + G) - (\phi^r + G)) \in L^{\infty}$.
In the mean time, we have an estimate
\begin{eqnarray*}
\| \phi^r \|_{\mathcal{W}^{2,p}(\Omega)} \le C \left ( \| G \|_{L^p(\Omega_s)} +
\| g -G \|_{W^{2-1/p,p}(\partial \Omega)}
+ \|\mathbf{F} \nabla G \|_{W^{1-1/p,p}(\Gamma)} \right ) .  \qed
\end{eqnarray*}
\end{proof}

\section{An Electrostatic Force Model and Some Estimates}
For the untransformed nonlinear Poisson-Boltzmann equation (\ref{PBE_1}) the electrostatic energy of the system
is defined \cite{Honig_90,Gilson_Eforce,Davis_Eforce} to be
\begin{eqnarray}
E  = \int_{\Omega} [\rho^f \phi - \frac{1}{2} \epsilon (\nabla
\phi)^2 - \kappa^2 (\cosh( \phi) -1)\chi ] dx, \label{Eng_def}
\end{eqnarray}
where the characteristic function $\chi=1$ in $\Omega_s$ and is $0$
in molecules $\Omega_{mf}, \Omega_{mr}$. This energy is very similar
to the energy functional defined in Eq.(\ref{eng_functional}), and
any potential function $\phi$ minimizing (\ref{eng_functional}) is
also the minimizer of this electrostatic energy because $\cosh(x)
\ge 1$. The function $\cosh( \phi) -1$ describes the physical fact
that the total electrostatic energy is zero when $\phi$ is
everywhere zero. The three terms in this energy represent three
types of energy densities, namely, the Coulomb energy, the
electrostatic stress energy and the osmotic stress energy of the
mobile ions. Based on this energy function, the following density
function of the force exerted on the molecule was derived
\cite{Davis_Eforce} by using a variational derivation method:
\begin{eqnarray}
\mathbf{f} = \rho^f \mathbf{E} - \frac{1}{2} |\mathbf{E}|^2 \nabla
\epsilon - \kappa^2 (\cosh(\phi)-1)  \nabla \chi,
\label{e_force_def}
\end{eqnarray}
where the three terms correspond to the Coulomb force, dielectric
pressure and the ionic pressure, respectively. The last two boundary
forces are always in the normal direction of the molecular surface
because of the gradients of $\epsilon$ and the characteristic
function $\chi$. The electric force defined in (\ref{e_force_def})
is physically justifiable, and can be converted into a form
identical to the Maxwell stress tensor(MST)
\cite{Gilson_Eforce,Jones_MST}. The MST describes the volume force
density in a linear dielectric, and has been widely utilized in
dielectrophoretic force and electrorotational torque calculations of
colloids, macromolecules and biological cells in continuous external
electric field \cite{Jones_MST}. In the context of interactions
between singular charges distribution and resulting singular
electric field, refinements are necessary to make this force model
computationally more tractable. Below we will discuss the treatments
of its three components.

The first term in Eq. (\ref{e_force_def}) might appear misleading because of the multiplication of two singular
functions, $\rho^f$ and $\mathbf{E}$, in its expression. We therefore would emphasis that at a singular change $x_i$
the electric potential field multiplied with $\rho^f$
in Eq.(\ref{Eng_def}) shall be interpreted as the summation of reaction potential field $\phi^r$, i.e.,
the regular component of the potential solution, and the Coulomb potential induced by all other singular
charges \cite{Davis_Eforce}:
\begin{eqnarray}
\rho^f \mathbf{E} = \sum_{i} q_i \delta(x_i) \mathbf{E} & := & \sum_{i}
q_i \delta(x_i) \nabla \left( \phi^r(x) + \sum_{j \ne i} G_j(x)
\right) \nonumber \\
 & =  & \sum_{i} q_i \nabla \left ( \phi^r(x_i) + \sum_{j \ne i}
G_j(x_i) \right ) \delta(x_i). \label{force_sing_def}
\end{eqnarray}
This verifies that the force exerted at each charged atom is finite.
The eliminated term $G_i(x_i)$ corresponds to the self-energy of the
singular charges \cite{Davis_Eforce}.

Nevertheless, the body force density $\rho^f \mathbf{E}$ itself is still unbounded at the center of
every charged atom where the charge density is singular,
indicating that this body force density does not belongs to $L^p(\Omega_{mf})$
hence does not fit the assumption on the body force in Theorem (\ref{thm_3Delasticity}).
An alternative model is therefore necessary to regularize
these singular body forces to ensure the solvability of the elasticity equation.
In this study, the singular the body force density is modeled by a Gaussian function
\begin{eqnarray}
\mathbf{f_b} & = & \sum_{i} a_i e^{-(x-x_i)^2/\sigma_i} \mathbf{n}_i, \label{f_body_2}
\end{eqnarray}
where the unit normal vector is aligned with the corresponding gradient in Eq.(\ref{force_sing_def});
the  decay parameter $\sigma_i$ is chosen such that
\begin{eqnarray}
\sum_{i} a_i e^{-R_i^2/\sigma_i} = \delta \label{sigma_def}
\end{eqnarray}
for a given sufficiently small number $\delta$, and $R_i$ is the van der Waals' radius of atom $i$.
This means that the Gaussian function is essentially compact supported in its associated atom.
The prefactor $a_i$ is determined by the conservation of force in each atom:
\begin{eqnarray}
\int_{atom_i} a_i e^{-(x-x_i)^2/\sigma_i} dx = q_i [ \phi^r(x_i) + \sum_{j \ne i} G_j(x_i) ] =
4\pi a_i \int_{0}^{R_i} r^2 e^{-r^2/\sigma_i} dr. \label{a_def}
\end{eqnarray}
The body force $\mathbf{f}_b$ modeled by this Gaussian is uniformly
continuous in $\Omega_{mf}$ and belongs to $L^p(\Omega_{mf})$ for
any $p>0$. Moreover, the lemma below proves that the difference of
two continuous body force densities also belongs to
$L^p(\Omega_{mf})$, and is small if the difference between two total
body forces which they approximate is small.
\begin{lemma} \label{lemma_fb_lp}
Let $A_1, A_2$ be two given numbers and $|A_1| \le P, |A_2| \le P$ for some $P$. Let
$$f_j = a_j e^{-(x-x_0)^2/\sigma_j} \quad \mbox{such that} \quad
 \int_{x-x_0 \le R} a_j e^{-(x-x_0)^2/\sigma_j}dx = A_j \quad \mbox{for} \quad j=1,2, $$
where $a_j,\sigma_j$ are determined from Eqs.(\ref{a_def},\ref{sigma_def}) for the same atom centered at $x_0$ and
of radius $R$. Then if $|A_1 - A_2| \le \delta'$ for some $\delta' >0$, we have
\begin{eqnarray}
\int_{x-x_0 \le R} |f_1 - f_2|^p dx \le C \delta' \label{fb_lp}
\end{eqnarray}
for some constant $C$ depending only on $R$ and $P$.
\end{lemma}
\begin{proof}
The prefactor $a$ and the decay rate $\sigma$ are uniformly continuous functions of $A$ for $|A| \le P$
if $$f=a e^{-(x-x_0)^2/\sigma}$$ is the approximation of $A$ as defined by the lemma.
But then there exists a constant $C$ depending on the derivatives of $a$ and $\sigma$ with respect to
$A$ such that $|f_1(x) - f_2(x)| \le C \delta'$ if $|A_1 - A_2| \le \delta'$. The conclusion of the lemma
follows directly. \qed
\end{proof}

The last two terms in Eq.(\ref{e_force_def}) represent the
electrostatic surface forces on the molecule. It is worth noting
that the second term is not well defined and is computationally
intractable if there is no dielectric boundary smoothing, due to the
discontinuous electric field $\mathbf{E}$ on the molecular surface
indicated by the interface condition
\begin{eqnarray*}
\epsilon_m \nabla \phi_m \cdot \mathbf{n} = \epsilon_s \nabla \phi_s \cdot \mathbf{n} \qquad \mbox{or}
\qquad \epsilon_m \mathbf{E}_m \cdot \mathbf{n} = \epsilon_s \mathbf{E}_s \cdot \mathbf{n}.
\end{eqnarray*}
To remove this ambiguity we consider a infinitesimal displacement
$h$ of the molecular surface in its out normal direction, see
Fig.(\ref{interface_move}). The change of the electrostatic stress
energy due to this small displacement is the work done by the
dielectric pressure along this displacement:
\begin{eqnarray*}
   \int_{\Omega_s' + \Omega_m'} -\frac{1}{2} \epsilon |\mathbf{E}|^2 dx
- \int_{\Omega_s + \Omega_m} - \frac{1}{2} \epsilon | \mathbf{E}|^2 dx & = & \int_{\Gamma_f \times h} -\frac{1}{2} (\epsilon_s |\mathbf{E}_s|^2 - \epsilon_m |\mathbf{E}_m |^2) dx \\
& = & h \int_{\Gamma_f} \mathbf{f}_e ds.
\end{eqnarray*}
This suggests the dielectric force density $\mathbf{f}_e$ is essentially the difference between
$-\frac{1}{2} \epsilon_s | \mathbf{E}_s|^2$ and $-\frac{1}{2} \epsilon_m | \mathbf{E}_m|^2$ on the dielectric interface, i.e.,
\begin{eqnarray}
\mathbf{f}_e =  -\frac{1}{2} (\epsilon_s | \mathbf{E}_s|^2 - \epsilon_m | \mathbf{E}_m|^2) \mathbf{n}. \label{fe_def}
\end{eqnarray}
By combining definitions (\ref{e_force_def}), (\ref{f_body_2}) and (\ref{fe_def}) we would obtain a complete
model of the electrostatic body force and surface force:
\begin{eqnarray}
\mathbf{f_b} & = & \sum_{i} a_i e^{-(x-x_i)^2/\sigma_i} \mathbf{n}_i, \label{fb_def} \\
\mathbf{f_s} & = &  -\frac{1}{2} (\epsilon_s | \mathbf{E}_s|^2 - \epsilon_m | \mathbf{E}_m|^2) \mathbf{n} -
\kappa^2 (\cosh(\phi) -1) \mathbf{n}. \label{fs_def}
\end{eqnarray}
\begin{remark}
In the sequel we will estimate the term $\| \mathbf{f}_s
\|_{W^{1-1/p,p}(\Gamma_f)}$. Although the term $\| \mathbf{E}_s
\|_{W^{1-1/p,p}(\Gamma_f)}$ can be directly related to $\| \phi_s
\|_{W^{2,p}(\Omega_s)}$ since the latter term is bounded in $\Omega_s$,
one can not estimate $\| \mathbf{E}_m \|_{W^{1-1/p,p}(\Gamma_f)}$
similarly by relating it with $\| \phi_m \|_{W^{2,p}(\Omega_{mf})}$
because $\phi_m$ contains singularities and hence is unbounded in
$\Omega_{mf}$. Instead we follow the procedure in the proof of (e,f)
in Lemma (\ref{lemma_green}) and eventually estimate this trace norm
of $\mathbf{E}$ in $\Omega_s^-$ which does not contain potential
singularities; the details are omitted due to similarity of these
two proofs.
\end{remark}
\begin{remark}
The surface force definition presented in Eq.(\ref{fs_def}) applies
only to the discontinuous dielectric model as adopted in this study.
In the continuous dielectric models, which are also widely used for
in the implicit solvent simulations, different surface force
definition will be derived \cite{Im_CPC98}. However, the analysis on the
electrostatic forces given in the below is also applicable to
general surface force function $f_s =
f_s(\mathbf{E}_s,\mathbf{E}_m,\phi)$, and might be simplified if
electrical field $\mathbf{E}$ is continuous, i.e., $\epsilon$ is
continuous on $\Gamma$.
\end{remark}

\begin{figure}[!ht]
\begin{center}
\includegraphics[width=7cm]{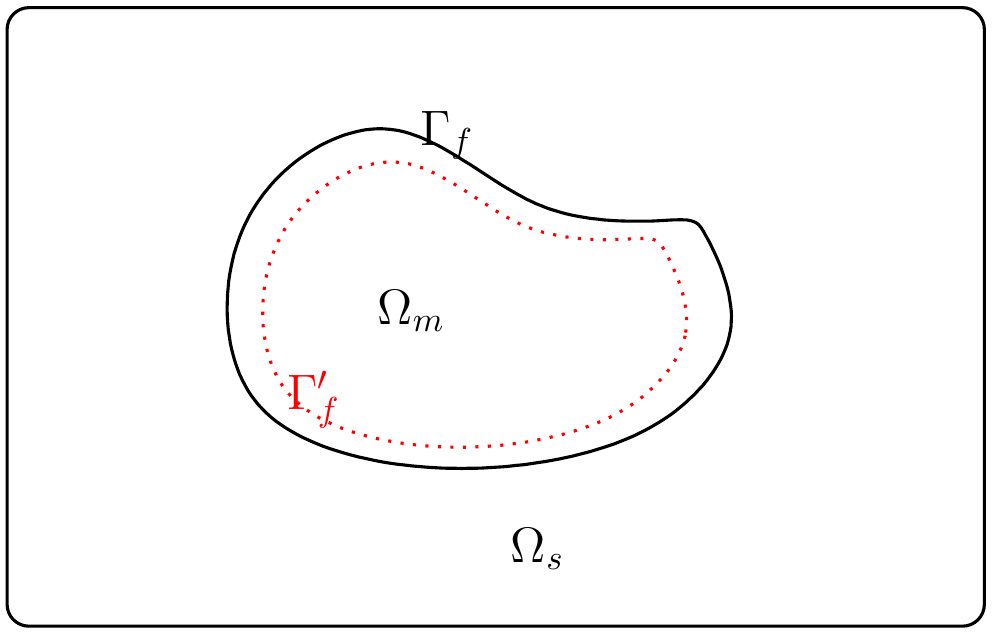}
\caption{Displacement of the molecular surface $\Gamma_f$. The solid
black line is the surface before displacement and the dashed red
line is the surface after displacement. The new solvent region
$\Omega_s'$ is $\Omega_s$ plus the strip between two surfaces; The
strip subtracted from $\Omega_m$ the equals the new solute region
$\Omega_m'$.} \label{interface_move}
\end{center}
\end{figure}

The electrostatic forces defined in Eq.(\ref{fb_def}) and
Eq.(\ref{fs_def}) are also subject to the Piola transformation.
Moreover these forces can not be directly supplied to the elasticity
equation; only the forces relative to a reference state can be
supplied. This is because a molecule is in an equilibrium state and
has no elastic deformation if the electrostatic potential is induced
only by the molecule itself and the solvent with physiological ionic
strength, in the absence of interactions with other molecules. We
refer to this state as the free state and use it as the reference
state. The {\it net} body force or the {\it net} surface force is
therefore defined to be the difference between that for a molecule
in non-free state and that for the same molecule in the free state.
To abuse the notation these differences are still referred to as the
body force and the surface force, and are denoted by $\mathbf{f}_b$
and $\mathbf{f}_s$ respectively:
\begin{eqnarray}
\mathbf{f}_b & := & \mathbf{f}_b - \mathbf{f}_{b0} \label{fb_def_2} \\
\mathbf{f}_s & := & \mathbf{f}_s - \mathbf{f}_{s0} \label{fs_def_2}
\end{eqnarray}
where $\mathbf{f}_{b0}$ and $\mathbf{f}_{s0}$ are the body force and
the surface force in the free state, and are constant vector fields for
any given macromolecule.

Physically, these two forces $\mathbf{f}_b$ and $\mathbf{f}_s$
shall be vanishing if there is no change of ionic strength and no additional molecules present,
and will be small for small change of ionic strength and weakly interacting additional molecules.
To reflect this physical reality and to facilitate the mathematical analysis,
we decompose (into four steps) the transition from the original single deformable molecule immersed in
aqueous solvent with physiological ionic strength to the final system with added rigid molecules,
varied ionic strength and deformed molecules.
In the first step, we change only the solvent from physiological ionic strength to the target strength,
and assume that the molecule $\Omega_{mf}$
does not have a conformational change although the {\it net} electrostatic force is not zero due to this
change of ionic strength. The electrostatic potential and forces at the end of the first perturbation
are denoted by $\phi_{1}$ and $\mathbf{f}_{b1}, \mathbf{f}_{s1}$, respectively. In the second step,
we alter the dielectric constant in the smooth domain $\Omega_{mr}$ from $\epsilon_s$ to $\epsilon_m$.
This low dielectric space represents the empty interior of the added molecules. The electrostatic potential
and forces after the second step are respectively denoted by $\phi_2$ and $\mathbf{f}_{b2}, \mathbf{f}_{s2}$.
In the third step we place the singular charges into $\Omega_{mr}$ and define the electrostatic
potential and forces to be $\phi_3$ and $\mathbf{f}_{b3}, \mathbf{f}_{s3}$.
In the last step we allow the Poisson-Boltzmann equation to couple with the elastic deformation
so that the system will arrive at the final state with electrostatic potential $\phi$ and forces $\mathbf{f}_b, \mathbf{f}_s$.
We write the {\it net} body force $\mathbf{f}_b$ and the {\it net} surface force as the summation of their four components
 \begin{eqnarray}
\mathbf{f}_b & = &  (\mathbf{f}_b - \mathbf{f}_{b3}) + (\mathbf{f}_{b3} - \mathbf{f}_{b2}) +
(\mathbf{f}_{b2} - \mathbf{f}_{b1}) + (\mathbf{f}_{b1} - \mathbf{f}_{b0} ) \label{fb_decomposition} \\
\mathbf{f}_s & = & (\mathbf{f}_s - \mathbf{f}_{s3}) + (\mathbf{f}_{s3} - \mathbf{f}_{s2}) +
(\mathbf{f}_{s2} - \mathbf{f}_{s1}) + (\mathbf{f}_{s1} - \mathbf{f}_{s0}) \label{fs_decomposition}
\end{eqnarray}
corresponding to the above decomposition, and estimate these components individually.

\subsection{The surface force due to changing ionic strength}
The electrostatic potential $\phi_0$ of the system in the free state is given by
\begin{eqnarray}
-\nabla \cdot (\epsilon \nabla \phi_0) + \kappa^2_0 \sinh(\phi_0) = \sum_{i}^{N_f} q_i \delta(x-x_i),  \label{Eq_phi0}
\end{eqnarray}
while the electrostatic potential $\phi_1$ after changing of the ionic strength satisfies
\begin{eqnarray}
-\nabla \cdot (\epsilon \nabla \phi_1) + \kappa^2 \sinh(\phi_1) = \sum_{i}^{N_f} q_i \delta(x-x_i). \label{Eq_phi1}
\end{eqnarray}
By subtracting Eq.(\ref{Eq_phi0}) from Eq.(\ref{Eq_phi1}) we get
\begin{eqnarray}
-\nabla \cdot (\epsilon \nabla \tilde{\phi}) + (\kappa^2 - \kappa^2_0) \cosh(\xi) \tilde{\phi} = (\kappa^2_0 - \kappa^2) \sinh(\phi_0) ~ \mbox{in} ~ \Omega, \label{Eq_phidiff1}
\end{eqnarray}
where $\tilde{\phi} = \phi_1 - \phi_0$ and $\xi(x) \in \left( \min \{\phi_1(x), \phi_0(x) \}, \max \{ \phi_1(x), \phi_0(x) \}  \right )$ is a function between
$\phi_1$ and $\phi_0$ satisfying the Cauchy mean value theorem
\begin{eqnarray*}
\sinh(\phi_1) = \sinh(\phi_0) + \cosh(\xi) (\phi_1 - \phi_0).
\end{eqnarray*}
We note that the singular charges disappear in Eq.(\ref{Eq_phidiff1}), and hence $\tilde{\phi} \in H^1(\Omega)$
and is also in $C^{\infty}$ in $\Omega_{mf}$ and $\Omega \setminus \overline{\Omega}_{mf}$. Moreover, following Theorem
(\ref{w2p_theorem}) we have the following $\mathcal{W}^{2,p}$ estimate for $\tilde{\phi}$:
\begin{eqnarray}
\| \tilde{\phi} \|_{\mathcal{W}^{2,p}(\Omega)} \le C \left( \| \tilde{\phi} \|_{L^p(\Omega)} + \| (\kappa^2_0 - \kappa^2) \sinh(\phi_0) \|_{L^p(\Omega)} + \| \tilde{G} \|_{L^p(\partial \Omega)}  \right), \label{phidiff1_estim_1}
\end{eqnarray}
where
\begin{eqnarray}
\tilde{G} = \sum_{i}^{N_f}  \frac{e^{-\kappa|x-x_i |} - e^{-\kappa_0 |x-x_i |}}{\epsilon_w |x - x_i |}  \approx
-(\kappa - \kappa_0)\sum_{i}^{N_f}  \frac{e^{-\kappa|x-x_i |}}{\epsilon_w} \label{gdiff_estim}
\end{eqnarray}
is the boundary condition of $\tilde{\phi}$ on $\partial \Omega$, and is the difference of boundary values of
$\phi_1$ and $\phi_0$. The approximation in Eq.(\ref{gdiff_estim}) is well defined for small
$(\kappa - \kappa_0)$. On the other hand, Lemma (\ref{linfty_lemma}) says that $\| \tilde{\phi} \|_{L^p(\Omega)}$ itself can
be estimate by
\begin{eqnarray}
 \| \tilde{\phi} \|_{L^p(\Omega)}  \le C \| \tilde{\phi} \|_{L^{\infty}(\Omega)}
\le C \left (  \| (\kappa^2_0 - \kappa^2) \sinh(\phi_0) \|_{L^2(\Omega)}
+ \| \tilde{G} \|_{H^{1/2}(\partial \Omega)}  \right ). \label{phidiff1_lp}
\end{eqnarray}
By combining Eqs. (\ref{phidiff1_estim_1}), (\ref{gdiff_estim}) and (\ref{phidiff1_lp}) we get
\begin{eqnarray}
\| \phi_1 - \phi_0 \|_{W^{2,p}(\Omega)} \le C |\kappa_0 - \kappa | \label{phidiff1_estim_2}
\end{eqnarray}
We now proceed to estimate the changes of electrostatic forces $\mathbf{f}_{b1} - \mathbf{f}_{b0},
\mathbf{f}_{s1} - \mathbf{f}_{s0}$. The body force change
\begin{eqnarray}
\| \mathbf{f}_{b1} - \mathbf{f}_{b0}\|_{L^p(\Omega_{mf})} & \le & C \sum_{i} | \tilde{\phi}(x_i) | \le C |\kappa_0 - \kappa |
\label{fb_diff_1}
\end{eqnarray}
follows from Lemma (\ref{lemma_fb_lp}). On the other hand,
\begin{eqnarray*}
\mathbf{f}_{s1} - \mathbf{f}_{s0} & = & - \frac{1}{2} \epsilon_s ( |\nabla \phi_{1s}|^2 - |\nabla \phi_{0s}|^2) \mathbf{n}  + \frac{1}{2} \epsilon_m ( |\nabla \phi_{1m}|^2 - |\nabla \phi_{0m}|^2) \mathbf{n} \nonumber \\
 & & - ~ \left (\kappa^2 (\cosh(\phi_{1s})-1) - \kappa^2_0 (\cosh(\phi_{0s})-1) \right) \mathbf{n}
 \nonumber \\
 & = & -\frac{1}{2} \epsilon_s (\nabla \phi_{1s} - \nabla \phi_{0s}) \cdot (\nabla \phi_{1s} + \nabla \phi_{0s}) \mathbf{n} \nonumber \\
& & + \frac{1}{2} \epsilon_m (\nabla \phi_{1m} - \nabla \phi_{0m}) \cdot (\nabla \phi_{1m} + \nabla \phi_{0m}) \mathbf{n}
 \nonumber \\
 & &  -~ \kappa^2 (\cosh(\phi_{1s}) - \cosh(\phi_{0s})) \mathbf{n} 
\nonumber \\
& & -~(\kappa^2 - \kappa^2_0) \cosh(\phi_{0s}) \mathbf{n}
\nonumber \\
& & -~ (\kappa^2 - \kappa_0^2).
\end{eqnarray*}
We note that with the mean value theorem, $\kappa^2 (\cosh(\phi_{1s}) - \kappa^2 \cosh(\phi_{0s}))$ can be related to the change of ionic strength as
\begin{eqnarray*}
\kappa^2 (\cosh(\phi_{1s}) - \cosh(\phi_{0s})) = \kappa^2 \sinh(\xi')(\phi_{1s} - \phi_{0s}).
\end{eqnarray*}
Moreover, we can not bound the trace norm of $ |\nabla \phi_{1m}|^2 - |\nabla \phi_{0m}|^2$ by its Sobolev norm
in subdomain $\Omega_{mf}$ where the singularities of the potential are located. Instead we follow the remark of Eq. (\ref{fs_def})
and estimate this term in subdomain $\Omega_s^-$. Thus the surface force change in the first perturbation step
can be estimated as
\begin{eqnarray}
\| \mathbf{f}_{s1} - \mathbf{f}_{s0} \|_{W^{1-1/p,p}(\Gamma_{f})} & \le &
C \Big ( \| \phi_{1} -  \phi_{0} \|_{W^{2,p}(\Omega_{s})} \|  \phi_{1} +  \phi_{0} \|_{W^{2,p}(\Omega_{s})}
\nonumber \\
 & & \quad + \|  \phi_{1} - \phi_{0} \|_{W^{2,p}(\Omega_s^-)}
                \|  \phi_{1} +  \phi_{0} \|_{W^{2,p}(\Omega_s^-)}  \nonumber \\
 & &  \quad + ~ \kappa^2 | \sinh(\xi') | \| \phi_{1} - \phi_{0}  \|_{W^{1,p}(\Omega_{s})}  \nonumber \\
& &  \quad + | \kappa^2 - \kappa^2_0| \|  \phi_{0s}  \|_{W^{1,p}(\Omega_{s})} + | \kappa^2 - \kappa^2_0| \Big ) \nonumber \\
& \le & C | \kappa - \kappa_0 | \Big ( \| \phi_{1} + \phi_{0}\|_{W^{1,p}(\Omega_{s})}  +
\| \phi_{1} + \phi_{0}\|_{W^{1,p}(\Omega_s^-)}  \nonumber \\
& &  \quad +~ \kappa^2 | \sinh(\xi') | +  ( \kappa + \kappa_0) ( \|  \phi_{0}  \|_{W^{1,p}(\Omega_{s})} + 1) \Big ) \nonumber \\
 & \le & C_s(\kappa) | \kappa - \kappa_0 |, \label{fs_diff_1}
\end{eqnarray}
where Lemma (\ref{w1p_algebra}) is applied to estimate the norm of the products of
two $W^{1,p}$ functions $\nabla \phi_{1} - \nabla \phi_{0}$ and $\nabla \phi_{1} + \nabla \phi_{0} $.

\subsection{The surface force due to adding a low dielectric constant cavity}
Although the variation of ionic strength will change the electrostatic potential of the
system, the magnitude of potential change is usually smaller than that induced by adding molecules
to the system. By adding molecules to the system we will not only have the additional singular charges
but also expand a cavity of low dielectric constant in the solvent. These two effects will be
considered separately, and this subsection estimates only the change of potential and forces due to the
additional cavity of low dielectric constant. The effect of added charges will be analyzed in the next subsection.

\begin{figure}[!ht] \label{fig_addspace}
\begin{center}
\includegraphics[width=12cm]{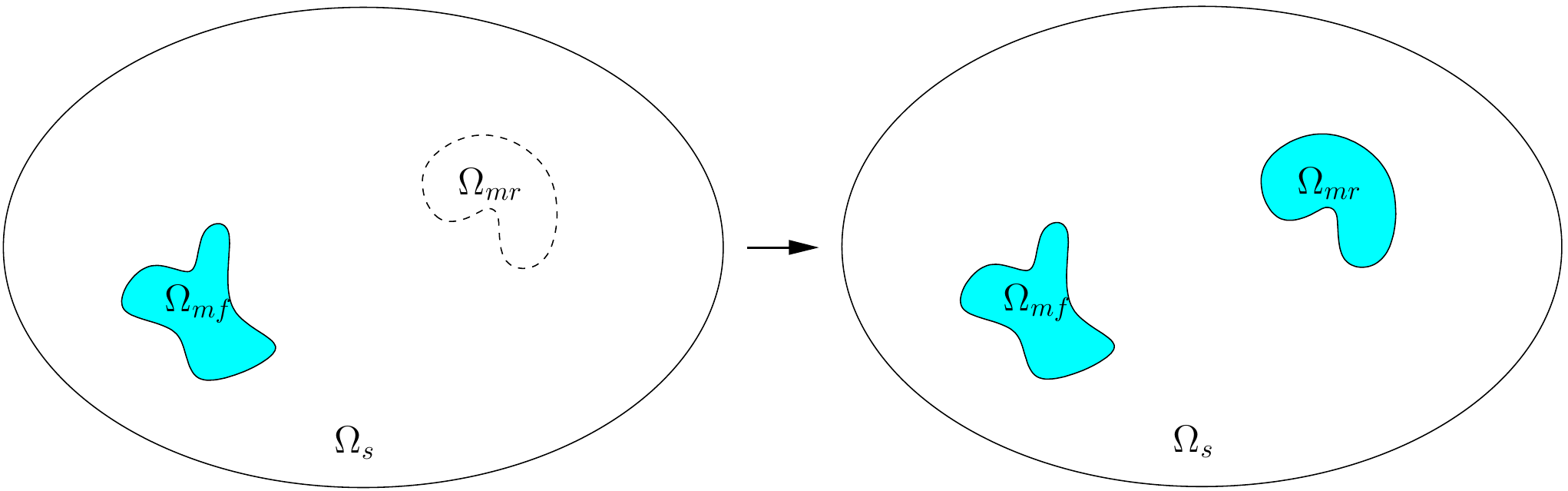}
\caption{Illustration of adding rigid molecule(s). Left: Before adding rigid molecules the domain $\Omega_{mr}$ is
occupied by solvent and hence has dielectric constant $\epsilon_s$. Right: After adding molecules the
domain $\Omega_{mr}$ has low dielectric constant $\epsilon_m$.}
\end{center}
\end{figure}

The electrostatic potential $\phi_2$ with an additional low dielectric cavity in the domain is described by
\begin{eqnarray}
-\nabla \cdot (\epsilon \nabla \phi_2) + \kappa^2 \sinh(\phi_2) = \sum_{i}^{N_f} q_i \delta(x-x_i) \label{Eq_phi2}
\end{eqnarray}
with the same boundary conditions as Eq.(\ref{Eq_phi1}). Here
the dielectric constant $\epsilon$ and ionic strength $\kappa$ are different from those in Eq.(\ref{Eq_phi1}),
and thus the subtraction of Eq.(\ref{Eq_phi1}) from
Eq.(\ref{Eq_phi2}) shall be individually conducted in $\Omega_{mf}, \Omega_{mr}$ and $\Omega_s$ to give the
following three equations:
\begin{eqnarray*}
-\nabla \cdot (\epsilon_m \nabla (\phi_2 - \phi_1)) & = & 0  ~ \mbox{in}~\Omega_{mf}, \\
-\nabla \cdot (\epsilon_s \nabla (\phi_2 - \phi_1)) + \kappa^2 (\sinh(\phi_2) - \sinh(\phi_1) & = & 0  ~ \mbox{in}~\Omega_{s},
\\
-\nabla \cdot (\epsilon_m \nabla \phi_2) + \nabla (\epsilon_s \nabla \phi_1) - \kappa^2 \sinh(\phi_1) & = & 0
~ \mbox{in}~\Omega_{mr}.
\end{eqnarray*}
By assembling these three equations we get a complete equation for
$\tilde{\phi} = \phi_2 - \phi_1$ in $\Omega$:
\begin{eqnarray}
-\nabla \cdot (\epsilon \nabla \tilde{\phi} ) + \kappa^2 (\sinh(\phi_2) - \sinh(\phi_1)) & = &
\frac{\epsilon_m}{\epsilon_s} \kappa^2 \sinh(\phi_1)  ,
\label{phidiff2_1}
\end{eqnarray}
where $\epsilon$ is the same as that in Eq.(\ref{Eq_phi2}), and the right-hand side is vanishing in $\Omega_s$ and $\Omega_{mf}$. This function (non-vanishing only in $\Omega_{mr}$) is equivalent to $- \nabla \cdot ((\epsilon_s - \epsilon_m) \nabla \phi_1) + \kappa^2 \sinh(\phi_1)$ since
\begin{eqnarray*}
-\nabla \cdot (\epsilon_s \nabla \phi_1) + \kappa^2 \sinh(\phi_1) & = & 0 ~\mbox{in}~ \Omega_{mr}.
\end{eqnarray*}
As before we notice that $\kappa^2 (\sinh(\phi_2) - \sinh(\phi_1))$ can be related to $\cosh(\xi) \tilde{\phi}$ with
a smooth function $\xi$ bounded by $\phi_1$ and $\phi_2$, and therefore $\tilde{\phi}$ in Eq.(\ref{phidiff2_1}) satisfies an estimate of the form
\begin{eqnarray*}
\| \tilde{\phi} \|_{\mathcal{W}^{2,p}(\Omega)}   & \le &  C \left ( \| \tilde{\phi} \|_{L^{p}(\Omega)} +
\|\frac{\epsilon_m}{\epsilon_s} \kappa^2 \sinh(\phi_1) \|_{L^{p}(\Omega_{mr})} \right )
\\
& \le & C  \| \sinh(\phi_1) \|_{L^{p}(\Omega_{mr})} \nonumber \\
& \le & C \| \sinh(\phi_1) \|_{L^{\infty}(\Omega)} \cdot V_{mr}
\end{eqnarray*}
which follows from Theorem (\ref{reg_LPBeq_theorem}) and Lemma (\ref{linfty_lemma}), considering that Eq.(\ref{Eq_phi2})
has a vanishing boundary condition. Here $\displaystyle{ V_{mr} }$ is the volume of $\Omega_{mr}$, suggesting
that $\displaystyle{ \| \tilde{\phi} \|_{\mathcal{W}^{2,p}(\Omega)} }$ can be made arbitrarily small by reducing
the volume of $\Omega_{mr}$.

The change electrostatic body force induced by this additional low dielectric cavity can be estimated as
\begin{eqnarray}
\| \mathbf{f}_{b2} - \mathbf{f}_{b1}\|_{L^p(\Omega_{mf})} & \le & C \sum_{i} | \tilde{\phi}(x_i) | \le C V_{mr} \label{fb_diff_2}.
\end{eqnarray}
The surface force change is
\begin{eqnarray*}
\mathbf{f}_{s2} - \mathbf{f}_{s1} & = &  - \frac{1}{2} \epsilon_s ( |\nabla \phi_{2s}|^2 - |\nabla \phi_{1s}|^2) \mathbf{n}
+ \frac{1}{2} \epsilon_m ( |\nabla \phi_{2m}|^2 - |\nabla \phi_{1m}|^2) \mathbf{n} \nonumber \\
& & - \kappa^2 (\cosh(\phi_{2s}) - \cosh(\phi_{1s}) \mathbf{n}
 \nonumber \\
 & = & - \frac{1}{2} \epsilon_s (\nabla \phi_{2s} - \nabla \phi_{1s}) \cdot (\nabla \phi_{2s} + \nabla \phi_{1s}) \mathbf{n} \nonumber \\
& &  + \frac{1}{2} \epsilon_m (\nabla \phi_{2m} - \nabla \phi_{1m}) \cdot (\nabla \phi_{2m} + \nabla \phi_{1m}) \mathbf{n} \nonumber \\
     & &  - \kappa^2 (\cosh(\phi_{2s}) - \cosh(\phi_{1s}))
     \mathbf{n},
\end{eqnarray*}
and thus can be estimated by
\begin{eqnarray}
\| \mathbf{f}_{s2} - \mathbf{f}_{s1} \|_{W^{1-1/p,p}(\Gamma_{f})} & \le &
C \Big ( \| \phi_{2} - \phi_{1} \|_{W^{1,p}(\Omega_{s})}  \| \phi_{2} + \phi_{1} \|_{W^{1,p}(\Omega_{s})}  \nonumber \\
& & + \| \phi_{2} - \phi_{1} \|_{W^{1,p}(\Omega_s^-)}  \| \phi_{2} + \phi_{1} \|_{W^{1,p}(\Omega_s^-)}
\nonumber \\
& &  + ~ \kappa^2 | \sinh(\xi') | \| \phi_{2} - \phi_{1}  \|_{W^{1,p}(\Omega_{s})} \Big ) \nonumber \\
& \le & C \| \phi_2 - \phi_1 \|_{\mathcal{W}^{2,p}(\Omega)} \le C \cdot V_{mr}, \label{fs_diff_2}
\end{eqnarray}
following from the similar arguments in last subsection for estimating $\mathbf{f}_{s1} - \mathbf{f}_{s0}$.

\subsection{The surface force due to additional singular charges}
In this subsection we will consider the change of electrostatic potential and force caused by
singular charges placed in the low dielectric space $\Omega_{mr}$. The low dielectric space $\Omega_{mr}$
with these charges completely models the rigid molecule which is expected to interact
with a flexible molecule $\Omega_{mf}$. The electrostatic potential field after this third perturbation
step satisfies the following equation
\begin{eqnarray}
-\nabla \cdot (\epsilon \nabla \phi_3) + \kappa^2 \sinh(\phi_3) = \sum_{i}^{N_f} q_i \delta(x_i) +
\sum_{j}^{N_r} q_j \delta(x_j), \label{Eq_phi3}
\end{eqnarray}
while the change of potential, $\tilde{\phi} =\phi_3 - \phi_2$ is the solution of the equation
\begin{eqnarray}
-\nabla \cdot (\epsilon \nabla \tilde{\phi}) + \kappa^2 \cosh(\xi) \tilde{\phi} =  \sum_{j}^{N_r} q_j \delta(x_j), \label{phidiff3_1}
\end{eqnarray}
which is obtained by subtracting Eq.(\ref{Eq_phi2}) from Eq.(\ref{Eq_phi3}). Here $\xi(x)$ is a smooth function
defined by the mean value expansion $\sinh(\phi_3) = \sinh(\phi_2) + \cosh(\xi) (\phi_3 - \phi_2)$. To facilitate the
regularity analysis of $\tilde{\phi}$ we define its singular component $\tilde{G}$, which solves
\begin{eqnarray}
-\nabla \cdot (\epsilon_m \nabla \tilde{G}) = \sum_{j}^{N_r} q_j \delta(x_j) \label{diff3_G}
\end{eqnarray}
and its regular component $\tilde{\phi}^r$, which is the solution of
\begin{eqnarray}
-\nabla \cdot (\epsilon \nabla \tilde{\phi}^r ) + \kappa^2 \cosh(\xi) \tilde{\phi}^r =
\nabla \cdot ((\epsilon - \epsilon_m) \nabla \tilde{G} ) - \kappa^2 \cosh(\xi) \tilde{G}.
\end{eqnarray}
It shall be noted that $\nabla \cdot ((\epsilon - \epsilon_m) \nabla \tilde{G} )$ is nonzero only on
the molecular surfaces $\Gamma_f$ and $\Gamma_r$, and can be represented as an interface condition
$(\epsilon_s -\epsilon_m) \nabla \tilde{G} \cdot \mathbf{n}$ on each of these two molecular surfaces
similar to that in Eq.(\ref{fG_def2}). We notice that
\begin{eqnarray}
\tilde{G}(x) = \sum_{j}^{N_r} \frac{q_j}{\epsilon_m |x-x_j |} \label{phi3_G}
\end{eqnarray}
and is of $C^{\infty}$ wherever away from any of $x_j$, hence of $C^{\infty}(\overline{\Omega}_s)$,
and thus so is $-(\epsilon_s -\epsilon_m) \nabla \tilde{G} \cdot \mathbf{n}$ on $\Gamma_f$ and $\Gamma_r$.
The $\mathcal{W}^{2,p}$ estimate of $\tilde{\phi}$ in $\Omega_s$ says that
\begin{eqnarray}
\| \tilde{\phi} \|_{\mathcal{W}^{2,p}(\Omega)} & \le & C \left( \kappa^2 \| \cosh(\xi) \tilde{G} \|_{L^p(\Omega_s)}
+ \| (\epsilon_s -\epsilon_m) \nabla \tilde{G} \|_{W^{1-1/p,p}(\Gamma_f)} \right . \nonumber \\
& & \left . + \| (\epsilon_s -\epsilon_m) \nabla \tilde{G} \|_{W^{1-1/p,p}(\Gamma_r)}
+ \| g \|_{W^{1-1/p,p}(\partial \Omega)} + \| \tilde{\phi} \|_{L^{p}(\Omega)}  \right)  \nonumber \\
& \le & C \left(  \| \tilde{G} \|_{L^p}  + \| \tilde{G} \|_{W^{2,p}(\Omega_s)} +
\| g \|_{W^{1,p}(\Omega_s)}  \right)
\rightarrow 0 ~\mbox{as}~q_j\rightarrow 0 \label{phidiff3_2}
\end{eqnarray}
where
$$ g = \sum_{j}^{N_r} q_j \frac{e^{-k|x-x_j|}}{\epsilon_m |x-x_j |}
~\mbox{on}~\partial \Omega $$
is the boundary condition of Eq.(\ref{phidiff3_1}), and is the difference of boundary conditions of Eq.(\ref{Eq_phi2})
and Eq.(\ref{Eq_phi3}).

We now analyze the change of the electrostatic forces due to the inclusion of additional singular charges.
For body force we have
\begin{eqnarray}
\| \mathbf{f}_{b3} - \mathbf{f}_{b2}\|_{L^p(\Omega_{mf})} & \le & C \sum_{i}
| \tilde{\phi}(x_i) + \sum_{j \ne i} \tilde{G}_j(xi)  | \rightarrow 0 ~\mbox{as}~q_j\rightarrow 0, \label{fb_diff_3}
\end{eqnarray}
and for the surface force we know
\begin{eqnarray*}
\mathbf{f}_{s3} - \mathbf{f}_{s2} & = & - \frac{1}{2} \epsilon_s ( |\nabla \phi_{3s}|^2 - |\nabla \phi_{2s}|^2) \mathbf{n}
+ \frac{1}{2} \epsilon_m ( |\nabla \phi_{3m}|^2 - |\nabla \phi_{2m}|^2) \mathbf{n}  \nonumber \\
& & - (\kappa^2 \cosh(\phi_{3s}) - \kappa^2 \cosh(\phi_{2s}) \mathbf{n} \nonumber \\
 & = & - \frac{1}{2} \epsilon_s (\nabla \phi_{3s} - \nabla \phi_{2s}) \cdot (\nabla \phi_{3s} + \nabla \phi_{2s}) \mathbf{n} \nonumber \\
& & + \frac{1}{2} \epsilon_m (\nabla \phi_{3m} - \nabla \phi_{2m}) \cdot (\nabla \phi_{3m} + \nabla \phi_{2m}) \mathbf{n}
\nonumber \\
 & &  - \kappa^2 (\cosh(\phi_{3s}) - \cosh(\phi_{2s})) \mathbf{n}.
\end{eqnarray*}
This surface force difference can then be estimated by
\begin{eqnarray}
\| \mathbf{f}_{s3} - \mathbf{f}_{s2} \|_{W^{1-1/p,p}(\Gamma_{f})} &\le &
C \Big ( \| \phi_{3s} -  \phi_{2s} \|_{W^{2,p}(\Omega_{s})} \| \phi_{3s} +  \phi_{2s} \|_{W^{2,p}(\Omega_{s})}  \nonumber \\
& &  + \| \phi_{3m} -  \phi_{2m} \|_{W^{2,p}(\Omega_{s}^-)} \| \phi_{3m} + \phi_{2m} \|_{W^{2,p}(\Omega_{s}^-)}
\nonumber \\
 & & + \kappa^2 | \sinh(\xi') | \| \phi_{3s} - \phi_{2s}  \|_{W^{2,p}(\Omega_{s})} \Big ) \nonumber \\
& \le & C  \Big ( \| \phi_{3} -  \phi_{2}\|_{\mathcal{W}^{2,p}(\Omega_{s})} +
\| \phi_{3} -  \phi_{2}\|_{\mathcal{W}^{2,p}(\Omega_{s}')} \Big )
\rightarrow 0 \nonumber \\
& & \mbox{as}~q_j\rightarrow 0 \label{fs_diff_3}
\end{eqnarray}
where the constant $C$ depends on the $\mathcal{W}^{2,p}$ norm of $\phi_2,\phi_3$, and therefore is bounded
if $\phi_2,\phi_3$ are bounded.

\subsection{The surface force due to molecular conformational change}
We now consider the change of electrostatic potential and surface
forces induced by elastic displacement. By subtracting
Eq.(\ref{Eq_phi3}) from Eq.(\ref{PBE_2}) and with a few algebraic
manipulations we get the governing equation for $\tilde{\phi} = \phi
- \phi_3$:
\begin{eqnarray}
-\nabla \cdot (\epsilon \mathbf{F} \nabla \tilde{\phi}) + J \kappa^2 \cosh(\xi) \tilde{\phi}
& = & (J-1)\sum_{i}^{N_f+N_r} q_i \delta(x_i) + \nabla \cdot (\epsilon (\mathbf{F} -\mathrm{I}) \nabla \phi_3)
\nonumber \\
& &  + (J-1) \kappa^2 \sinh(\phi_3), \label{Eq_phidiff4}
\end{eqnarray}
where the function $\xi$ is defined by use of the mean value expansion
$\sinh(\phi) = \sinh(\phi_3) + \cosh(\xi) (\phi - \phi_3)$. Unlike its counterparts in the analysis for the
first two steps, this function $\xi$ is not piecewise smooth since $\phi^r$ of Eq.(\ref{PBE_2}) belongs to $\mathcal{W}^{2,p}(\Omega)$
hence is only piecewise uniformly differentiable. The resulting mean value function $\xi$ is therefore a piecewise uniformly
continuous function. Because of the appearance of
remaining singular charges in the right hand side of Eq.(\ref{Eq_phidiff4}), we know that $\tilde{\phi}$ is
not in $H^1$ globally. Again we employ the decomposition $\tilde{\phi} = \tilde{G}_f + \tilde{G}_r +
\tilde{\phi}^r$ to separate the singular components $\tilde{G}_f, \tilde{G}_r$ and the regular component $\tilde{\phi}^r$. The first singular component $\displaystyle{ \tilde{G}_f= \sum_{j}^{N_f} \tilde{G}_{fj} }$ is induced by all $N_f$ singular charges in $\Omega_{mf}$
\begin{eqnarray}
-\nabla \cdot (\epsilon_m \mathbf{F} \nabla \tilde{G}_f) & = & (J-1)\sum_{i}^{N_f} q_i \delta(x_i),
\end{eqnarray}
while the second singular component $\displaystyle{ \tilde{G}_r = \sum_{j}^{N_r} \tilde{G}_{rj} } $ is caused by all $N_r$ singular
charges in $\Omega_{mr}$
\begin{eqnarray}
-\nabla \cdot (\epsilon_m \mathbf{F} \nabla \tilde{G}_r) & = & (J-1)\sum_{j}^{N_r} q_j \delta(x_j),
\end{eqnarray}
and both singular components have estimates similar to Eqs.(\ref{green_estim}) and (\ref{green_grad_estim})
\begin{eqnarray}
\| \tilde{G}_f \|_{L^{\infty}(\Omega_s)}  \le
 \frac{ \|J-1\|_{L^{\infty}(\Omega)}  N_f K q_{max}}{\delta_f} ~\mbox{in}~\overline{\Omega}_s, \label{green_estim_2} \\
\| \tilde{G}_r \|_{L^{\infty}(\Omega_s)}  \le
 \frac{\|J-1\|_{L^{\infty}(\Omega)} N_r K q_{max}}{\delta_r} ~\mbox{in}~\overline{\Omega}_s, \label{green_estim_3} \\
\| \nabla \tilde{G}_f \|_{L^{\infty}(\Omega_s)} \le
\frac{\|J-1\|_{L^{\infty}(\Omega)} N_f K q_{max}}{\delta_f^2} ~\mbox{in}~\overline{\Omega}_s, \label{green_grad_estim_4} \\
\| \nabla \tilde{G}_r \|_{L^{\infty}(\Omega_s)} \le
\frac{\|J-1\|_{L^{\infty}(\Omega)} N_r K q_{max}}{\delta_r^2} ~\mbox{in}~\overline{\Omega}_s. \label{green_grad_estim_5}
\end{eqnarray}
By subtracting the singular components $\tilde{G}_f,\tilde{G}_r$ from Eq.(\ref{Eq_phidiff4}) we obtain an equation for the
regular component
\begin{eqnarray}
-\nabla \cdot (\epsilon \mathbf{F} \nabla \tilde{\phi}^r) + J \kappa^2 \cosh(\xi) \tilde{\phi}^r
& = & \nabla \cdot (\epsilon (\mathbf{F} -\mathrm{I}) \nabla \phi_3) + (J-1) \kappa^2 \sinh(\phi_3) \nonumber \\
&  & - \nabla \cdot ((\epsilon -\epsilon_m) \mathbf{F} \nabla \tilde{G}_f) \nonumber \\
& & - \nabla \cdot ((\epsilon -\epsilon_m) \mathbf{F} \nabla \tilde{G}_r), \label{Eq_phirdiff4}
\end{eqnarray}
where the last two items $\nabla \cdot ((\epsilon -\epsilon_m) \mathbf{F} \nabla \tilde{G}_f)$ and
$\nabla \cdot ((\epsilon -\epsilon_m) \mathbf{F} \nabla \tilde{G}_r)$
prescribe two interface conditions on the molecular surfaces $\Gamma_f$ and $\Gamma_r$:
\begin{eqnarray}
f_{G_f} & = & (\epsilon_s - \epsilon_m) \mathbf{F} \nabla \tilde{G}_f \cdot \mathbf{n},
\label{fGf} \\
f_{G_r} & = & (\epsilon_s - \epsilon_m) \mathbf{F} \nabla \tilde{G}_r \cdot \mathbf{n},
\label{fGr}
\end{eqnarray}
similar to that defined in Eq.(\ref{fG_def2}). For the regular component $\tilde{\phi}^r$,
Theorem (\ref{reg_LPBeq_theorem}) states that it can be estimated with respect to the $\mathcal{W}^{2,p}$ norm as follows
\begin{eqnarray}
\| \tilde{\phi}^r \|_{\mathcal{W}^{2,p}(\Omega)} & \le & C \left( \| \tilde{\phi}^r \|_{L^p(\Omega)} + \| f_{G_f} \|_{W^{1-1/p,p}(\Gamma_f)} + \| f_{G_r} \|_{W^{1-1/p,p}(\Gamma_f)} \right . \nonumber \\
& & +~ \| f_{G_f} \|_{W^{1-1/p,p}(\Gamma_r)} +~\| f_{G_r} \|_{W^{1-1/p,p}(\Gamma_r)} \nonumber  \\
& & +~ \left. \| (J-1) \kappa^2 \sinh(\phi_3) \|_{L^p(\Omega)}
 + \| \nabla \cdot (\epsilon (\mathbf{F} -\mathrm{I}) \nabla \phi_3) \|_{L^p(\Omega)} \right) \nonumber \\
 & \le & C \left( \| f_{G_f} \|_{W^{1-1/p,p}(\Gamma_f)} + \| f_{G_r} \|_{W^{1-1/p,p}(\Gamma_f)} \right .
\nonumber  \\
& & +~ \| f_{G_f} \|_{W^{1-1/p,p}(\Gamma_r)} + \| f_{G_r} \|_{W^{1-1/p,p}(\Gamma_r)}  \nonumber \\
& & +~ \left. \| (J-1) \sinh(\phi_3) \|_{L^p(\Omega)} +
 \| \mathbf{F}- \mathrm{I} \|_{W^{1,p}(\Omega)} \| \phi_3 \|_{W^{2,p}(\Omega)} \right) \nonumber \\
 & \le &  C\left( \|J-1 \|_{W^{1,p}(\Omega_s)} \| \mathbf{F} \|_{W^{1,p}(\Omega_s)} \right.
\\
& & +~ \|J-1\|_{L^{\infty}(\Omega)} \| \sinh(\phi_3) \|_{L^p(\Omega)} \nonumber \\
 & & +~ \left.
 \| \mathbf{F}-\mathrm{I} \|_{W^{1,p}(\Omega)} \| \phi_3 \|_{\mathcal{W}^{2,p}(\Omega) } \right ) \nonumber \\
& = & C(\| J-1 \|_{W^{1,p}(\Omega_s)} + \| J-1\|_{L^{\infty}(\Omega_s)} + \| \mathbf{F} - \mathrm{I} \|_{W^{1,p}(\Omega)}),
\label{phidiff4_estim}
\end{eqnarray}
where in the last inequality we applied the estimate in Eq.(\ref{fG_estimate}) for the interface conditions
$f_{G_f}$ and $f_{G_r}$. Finally, we estimate the change of the electrostatic forces due to the elastic deformation.
By definition, the body force change is attributed to the variation of regular
component(reaction field) $\tilde{\phi}^r$ and the variations of the the singular components(Coulomb potential field),
and thus can be estimated as:
\begin{eqnarray}
\hspace*{-0.7cm}
\| \mathbf{f}_{b} - \mathbf{f}_{b3}\|_{L^p(\Omega_{mf})} & \le & C \sum_{i}
| \tilde{\phi}(x_i) + \sum_{j \ne i} \tilde{G}_{fj}(x_i) + \sum_{j} \tilde{G}_{rj}(x_i) |  \nonumber \\
& \le & C(\| J-1 \|_{W^{1,p}(\Omega_s)} + \| J-1\|_{L^{\infty}(\Omega_s)} + \| \mathbf{F} - \mathrm{I} \|_{W^{1,p}(\Omega)}).
\label{fb_diff_4}
\end{eqnarray}
The surface force change in this step is defined to be
\begin{eqnarray}
\mathbf{f}_s - \mathbf{f}_{s3} & = & - \frac{1}{2} \epsilon_s ( |\mathbf{F} \nabla \phi_{s}|^2 - |\nabla \phi_{3s}|^2) \mathbf{n}
+ \frac{1}{2} \epsilon_m ( |\mathbf{F} \nabla \phi_{m}|^2 - |\nabla \phi_{3m}|^2) \mathbf{n} \nonumber \\
 & & - \kappa^2 (\cosh(\phi_{s}) - \kappa^2 \cosh(\phi_{3s}) \mathbf{n}
 \nonumber \\
 & = & - \frac{1}{2} \epsilon_s (\mathbf{F} \nabla \phi_{s} - \nabla \phi_{3s}) \cdot (\mathbf{F} \nabla \phi_{s} + \nabla \phi_{3s}) \mathbf{n} \nonumber \\
& & + \frac{1}{2} \epsilon_m (\mathbf{F} \nabla \phi_{m} - \nabla \phi_{3m}) \cdot (\mathbf{F} \nabla \phi_{m} + \nabla \phi_{3m}) \mathbf{n}
\nonumber \\
 & &  - \kappa^2 \sinh(\xi') (\phi_{s} - \phi_{3s}) \mathbf{n}. \label{f_diff4_1}
\end{eqnarray}
It follows that
\begin{eqnarray}
\| \mathbf{f}_s - \mathbf{f}_{s3} \|_{W^{1-1/p,p}(\Gamma_{f})} & \le &
C \Big ( \| \mathbf{F} \nabla \phi - \nabla \phi_3 \|_{W^{1,p}(\Omega_s) } +
\| \mathbf{F} \nabla \phi - \nabla \phi_3 \|_{W^{1,p}(\Omega_s') } \nonumber  \\
&  & + \| \phi - \phi_3 \|_{W^{1,p}(\Omega_s)} \Big). \label{f_diff4_2}
\end{eqnarray}
To relate the estimate of $\mathbf{F} \nabla \phi - \nabla \phi_3$ to that of $ \phi - \phi_3 $ (the latter has already been estimated in Eq.(\ref{phidiff4_estim})), we make use of the relation
\begin{eqnarray*}
\| \mathbf{F} \nabla \phi - \nabla \phi_3 \|_{W^{1,p}(\Omega_s)} & = &
\| \mathbf{F} \nabla \phi - \mathbf{F} \nabla \phi_3  +  \mathbf{F} \nabla \phi_3 - \nabla \phi_3 \|_{W^{1,p}(\Omega_s)}   \nonumber \\
& \le & \| \mathbf{F} (\nabla \phi - \nabla \phi_3)\|_{W^{1,p}(\Omega_s)} +
\| (\mathbf{F}-\mathrm{I})\nabla \phi_3  \|_{W^{1,p}(\Omega_s)} \nonumber \\
& \le & \| \mathbf{F}\|_{W^{1,p}(\Omega_s)} \|(\nabla \phi - \nabla \phi_3)\|_{W^{1,p}(\Omega_s)} +
\nonumber \\
&  & \| (\mathbf{F}-\mathrm{I}) \|_{W^{1,p}(\Omega_s)} \| \phi_3  \|_{W^{2,p}(\Omega_s)}
\end{eqnarray*}
and a similar relation for $\| \mathbf{F} \nabla \phi - \nabla \phi_3 \|_{W^{1,p}(\Omega_s')}$. By collecting these results together
we can conclude from Eq.(\ref{f_diff4_2}) that
\begin{eqnarray}
\| \mathbf{f}_s - \mathbf{f}_{s3} \|_{W^{1-1/p,p}(\Gamma_{f})} & \le &
C \Big (  \| J - 1\|_{W^{1,p}(\Omega)} + \| \mathbf{F} - \mathrm{I} \|_{W^{1,p}(\Omega)} \Big ), \label{fs_diff_4}
\end{eqnarray}
which indicates the dependence and the boundedness of this electrostatic force component with respect to
the elastic displacement field.
\subsection{Complete estimation of the electrostatic forces}
The complete estimation of the electrostatic surface force is presented this lemma:
\begin{lemma}
The electrostatic force can be made arbitrary small by reducing the variations of ionic strength, the
volume of the additional low dielectric space, the added singular charge and the magnitude
of the elastic deformation.
\end{lemma}
\begin{proof}
Following from its decomposition schemes (\ref{fb_decomposition},\ref{fs_decomposition}), the estimation of
total electrostatic body force and surface fore can be readily completed by combining their respective four
components estimated in the four subsections above. The estimates for these two forces have an identical form
\begin{eqnarray}
\| \mathbf{f}_b - \mathbf{f}_{b0} \|_{L^p(\Omega_{mf})} & \le & C \Big ( |\kappa - \kappa_0| +
\| \phi_{2} -  \phi_{1}\|_{W^{2,p}(\Omega_{s})} + \| \phi_{2} -  \phi_{1}\|_{W^{2,p}(\Omega_{s}')} \nonumber \\
 & & \qquad +~ V_{mr} + \| \mathbf{F} -\mathrm{I} \|_{W^{1,p}(\Omega)} +
      \| J-1\|_{W^{1,p}(\Omega)} \Big  ), \label{fb_final} \\
\hspace*{-0.5cm}
\| \mathbf{f}_s - \mathbf{f}_{s0} \|_{W^{1-1/p,p}(\Gamma_{f})} & \le & C \Big ( |\kappa - \kappa_0| +
\| \phi_{2} -  \phi_{1}\|_{W^{2,p}(\Omega_{s})} + \| \phi_{2} -  \phi_{1}\|_{W^{2,p}(\Omega_{s}')} \nonumber \\
 & & \qquad + ~V_{mr} + \| \mathbf{F} -\mathrm{I} \|_{W^{1,p}(\Omega)} +
      \| J-1\|_{W^{1,p}(\Omega)} \Big  ). \label{fs_final}
\end{eqnarray}
It is noticed in Eq.(\ref{phidiff3_2}) that both $\| \phi_{2} -  \phi_{1}\|_{W^{2,p}(\Omega_{s})}$ and
$ \| \phi_{2} -  \phi_{1}\|_{W^{2,p}(\Omega_{s}')}$ can be made arbitrarily small by adjusting the charges of
added molecule $\Omega_{mr}$. Moreover
$\mathbf{F}(\mathbf{0})(x) = 0, J(\mathbf{0})(x)=1$ follow from their definitions and both
functions are infinitely differentiable in the neighborhood of each function in
$$X_p = \{ \mathbf{u} \in W^{2,p}(\Omega_{mf}) | \| \mathbf{u} \|_{W^{2,p}(\Omega_{mf})} \le M \}.$$
Applying the Taylor inequality we have
\begin{eqnarray*}
\| \mathbf{F} - \mathrm{I} \|_{W^{1,p}(\Omega_s)} & \le & \|| D\mathbf{F} | \| \| \mathbf{u} \|_{W^{2,p}(\Omega_s)}
 \\
\| J - 1 \|_{L^{\infty}} & \le & \|| D\mathbf{J} | \| \| \mathbf{u} \|_{W^{2,p}(\Omega_s)},
\end{eqnarray*}
hence the last two items in estimates (\ref{fb_final}, \ref{fs_final}) are also small for properly chosen $X_p$. \qed
\end{proof}

\section{Main Results: Existence of Solutions to the Coupled System} \label{CoupledSystem}
We now establish the main existence result in the paper.
It is noticed that for every element $\mathbf{v} \in X_p$ one can derive
a Piola transformation and solve for a unique potential solution of the
Poisson-Boltzmann equation with this Piola transformation.
The electrostatic forces computed from this potential solution belongs
to $W^{1-1/p,p}(\Gamma_f)$ hence there is also a unique solution
$\mathbf{u}$ to Eq.(\ref{Eq_elasticity}) corresponding to these electrostatic
forces.
This loop defines a map $S$ which associates every $\mathbf{v}$ with a new
displacement function $\mathbf{u}$.
Our existence result is based on the following version of the
Schauder fixed-point theorem.
\begin{theorem}\label{schauder}
Let $X_p$ be a closed convex set in a Banach space $X$ and let $S$ be
a continuous mapping of $X_p$ into itself such that the image of
$S(X_p)$ is relatively compact.
Then $S$ has a fixed-point in $X_p$.
\end{theorem}
\begin{proof}
See \cite{Zeid91a}.
\end{proof}
The Schauder Theorem depends on establishing continuity and
compactness of the map $S:X_p \rightarrow X_p$.
We notice that $X_p$ is convex and is weakly compact in $W^{2,p}$.
Therefore the mapping $S$ has at least one fixed-point in $X_p$ if we can
verify that $S$ is continuous in some weak topology.
\begin{theorem} \label{s_continuous}
$S: X_p \rightarrow X_p$ is weakly continuous in $W^{2,p}(\Omega_{mf})$.
\end{theorem}
\begin{proof} This proof follows the similar arguments in \cite{Grandmont_3Dcouple}.
Let $\mathbf{v}_n$ be a sequence in $X_p$ and $\mathbf{v}_n \weakconvg \mathbf{v}$ in $W^{2,p} $ as $n \rightarrow \infty$.
With these displacement fields, we can compute the electrostatic potential $\phi_n = \phi(\mathbf{v}_n)$(
hence the electrostatic body force $\mathbf{f}_{bn} = \mathbf{f}_b(\mathbf{v_n})$ and surface force
$\mathbf{f}_{sn} = \mathbf{f}_s(\mathbf{v_n})$) and
new displacement fields $\mathbf{u}_n = \mathbf{u}(\mathbf{v}_n)$ defining the mapping $S$. We know from
(\ref{fb_final}), (\ref{fs_final}) and (\ref{estimate_u})
that $\phi_n,\mathbf{f}_{bn},\mathbf{f}_{sn}$ and $\mathbf{u}_n$ are bounded independently of $n$. Therefore there exists a subsequence
$\mathbf{v}_{n_l} \subset \mathbf{v}_n$, an electrostatic potential $\bar{\phi}$, and a displacement field
$\bar{\mathbf{u}}$, such that
\begin{eqnarray*}
\phi(\mathbf{v}_{n_l}) \weakconvg \bar{\phi} ~\mbox{as}~ l \rightarrow \infty \\
\mathbf{u}(\mathbf{v}_{n_l}) \weakconvg \bar{\mathbf{u}} ~\mbox{as}~ l \rightarrow \infty
\end{eqnarray*}
We shall prove that $\mathbf{u}(\mathbf{v}) = \bar{\mathbf{u}}$ by investigating the limit of
the equations for $\phi(\mathbf{v}_{n_l})$ and $\mathbf{u}(\mathbf{v}_{n_l})$,
and of the expression for $\mathbf{F}(\mathbf{v}_{n_l}),\mathbf{f}_b(\mathbf{v}_{n_l}), \mathbf{f}_s(\mathbf{v}_{n_l})$
and $J(\mathbf{v}_{n_l})$. Since $\Phi(\mathbf{v}_{n_l}) \weakconvg \Phi(\mathbf{v})$ in the same weak topology
as $\mathbf{v}_{n_l} \weakconvg \mathbf{v}$ and $W^{2,p}$ is compactly embedded in $C^1$, there is a subsequence of
$\mathbf{v}_{n_k} \subset \mathbf{v}_{n_l}$ such that
\begin{eqnarray*}
\mathbf{v}_{n_k} \rightarrow \mathbf{v} ~\mbox{in}~C^1(\Omega)~\mbox{as}~ k \rightarrow \infty, \\
\Phi(\mathbf{v}_{n_k}) \rightarrow \Phi(\mathbf{v}) ~\mbox{in}~C^1(\Omega)~\mbox{as}~ k \rightarrow \infty;
\end{eqnarray*}
hence
\begin{eqnarray*}
J(\mathbf{v}_{n_k}) \rightarrow J(\mathbf{v}) ~\mbox{in}~C^0(\Omega)~\mbox{as}~ k \rightarrow \infty,
\end{eqnarray*}
following the definition of $J(\mathbf{v})$. The convergence of
\begin{eqnarray*}
\mathbf{F}(\mathbf{v}_{n_k}) \rightarrow \mathbf{F}(\mathbf{v}) ~\mbox{in}~C^0(\Omega)~\mbox{as}~ k \rightarrow \infty,
\end{eqnarray*}
which involves the inversion of $\nabla \Phi(\mathbf{v})$, is substantiated by continuous mapping
from a $n \times n$ matrix to its inverse in $C^0(\Omega)$, i.e.,
\begin{eqnarray*}
A_{n \times n} \in C^0(\Omega) \mapsto A^{-1}_{n \times n} \in C^0(\Omega)
\end{eqnarray*}
in the neighborhood of each invertible matrix of $C^0(\Omega)$, and by the invertibility of
$\nabla \Phi(\mathbf{v}_{n_k})$ in $W^{1,p}(\Omega)$. Now we can pass the equations satisfied by
$\phi_{n_k}$ and $\mathbf{v}_{n_k}$ to the limit and deduce that
\begin{eqnarray*}
\phi(\mathbf{v}_{n_l}) \weakconvg \phi(\mathbf{v}) =  \bar{\phi}, \\
\mathbf{u}_{n_l} \weakconvg \mathbf{u}(\mathbf{v}) = \bar{\mathbf{u}}.
\end{eqnarray*}
This proves the continuity of mapping $S$ in the weak topology of $W^{2,p}$. \qed
\end{proof}
Finally we verify that $S(X_p) \subset X_p$. By connecting the force estimates (\ref{fb_final})--(\ref{fs_final}) and
the estimate of displacement $\mathbf{u}$ in theorem (\ref{thm_3Delasticity}) we observe that
\begin{eqnarray}
\| \mathbf{u} \|_{W^{2,p}} & \le & C \Big ( \| \mathbf{f}_b \|_{L^p(\Omega_{mf})} +
\| \mathbf{f}_s \|_{W^{1-1/p,p}(\Gamma_f)}  \Big ) \nonumber \\
  & \le & C \Big ( |\kappa - \kappa_0|
\| \phi_{2} -  \phi_{1}\|_{W^{2,p}(\Omega_{s})} +  \| \mathbf{F} -\mathrm{I} \|_{W^{1,p}(\Omega_s)} +
      \| J-1\|_{L^{\infty}} \Big )  \nonumber \\
& \le & M \le \frac{C}{\max \{ C1,C2 \} } \label{u_bounded}
\end{eqnarray}
for appropriately small change in ionic strength and in the charges in the added
molecules, where $C,C1,C2$ are the constants prescribed in inequality (\ref{M_bound_1}).
Thus we verified that $S(X_p) \subset X_p$ and $\Phi(\mathbf{u})$ is invertible.
This gives the main result in the paper as the following theorem.
\begin{theorem}
There exists a solution to the coupled nonlinear PDE system (\ref{Eq_elasticity}) and (\ref{PBE_2}) for sufficiently small $\kappa - \kappa_0$ and sufficiently small rigid molecule $\Omega_{mr}$ with sufficiently small charges.
\end{theorem}
\begin{proof}
This follows from
Theorem~\ref{schauder} combined with Theorem~\ref{s_continuous}.
\end{proof}

\section{Variational Principle for Existence and/or Uniqueness}
In addition to the fixed point arguments, variational principles and
quasivariational inequalities are also widely used for analyzing
coupled systems of PDEs arising from multiphysics modeling. While
quasivariational inequalities are exclusively used for systems with
boundary conditions given by inequalities, a single
energy functional for the entire system is generally required for
the application of either of these two approaches, and the stationary point
of this energy functional with respect to each function shall produce the
corresponding differential equations
and all boundary conditions. This energy functional is usually given by
the total potential energy of the system, or by the sum of the potential energies
of each equation if these energies are compatible. While it remains
challenge to construct the total energy for our problem, we can give a
coupled weak form of the entire system:
\begin{eqnarray}
(Ax,y) = (\mathbf{f}_b,\mathbf{v})_{L^2(\Omega_{mf})} + (\mathbf{f}_s,\mathbf{v})_{L^2(\Gamma_{mf})}~
\forall~ y \in P, \label{varia_form}
\end{eqnarray}
where $x=(\mathbf{u},\phi), y = (\mathbf{v}, \psi)$ are in the product space $P$ of
$W^{2,p}(\Omega_{mf})$ for the displacement field $\mathbf{u}$ and the $W^{2,p}(\Omega)$
for the regular component of electrostatic potential
$\phi^r$, i.e., $P=W^{2,p}(\Omega_{mf}) \times W^{2,p}(\Omega)$, and the operator $A$ is
defined by
\begin{eqnarray}
(Ax,y) & = & (\mathbf{T}(\mathbf{u}), \mathbf{E}(\mathbf{v})) +
(\epsilon \mathbf{F} \nabla \phi, \nabla \phi) + (J \kappa^2 \sinh(\phi+G), \psi), \label{operator_A}
\end{eqnarray}
where the stress tensor $\mathbf{T}$ and the strain tensor $\mathbf{E}$ were given in
Eq. (\ref{Stru.eq.1}).

Unlike the piezoelectric problems to which variational principles
and quasivariational inequalities can be readily applied, we lack the
coupling of the electrostatic potential and elastic displacement at the level of
constitutive relations of the material \cite{Sofonea_piezo04,Han_piezo07}.
Instead, our electro-elastic coupling is through the electrostatic forces.
We note that variational principles have been formulated for a class of
fluid-solid interaction systems \cite{Xing_FS97}, which resemble our problem
in that the coupling is through the boundary conditions of the elasticity equation
instead of the constitutive relations. This will be examined for our problem in a future
work.

\section{Concluding Remarks}
In this paper we have proposed and carefully analyzed a nonlinear elasticity
model of deformation in macromolecules induced by electrostatic forces.
This was accomplished by coupling the nonlinear Poisson-Boltzmann equation
for the electrostatic potential field to the nonlinear elasticity equations
for elastic deformation.
The electrostatic of this coupled system is desribed by an
implicit solvation model, and the Piola transformation defined by the solution of
the elasticity equation is introduced into the Poisson-Boltzmann equation such
that both equations can be analyzed together in a undeformed configuration.
A key technical tool for coupling the two models is the use of an
harmonic extension of the elastic deformation field into the solvent
region of the combined domain.
Combining this technical tool with regularization techniques established in
\cite{Chen_rpbe} and a standard bootstrapping technique,
we showed that the Piola-transformed Poisson-Boltzmann equation is also
well-defined and the regular component of its solution has a
piecewise $W^{2,p}$-regularity.
This regularity matches that of the elastic deformation, giving access to
a Schauder fixed-point theorem-based analysis framework for rigorously
establishing the existence of solutions to this coupled nonlinear PDE
system for small perturbation in the ionic strength and for small added
charges.
The existence of large deformation for large perturbations in ionic strength
and/or charges can be obtained by combining our local result with general
continuation techniques for nonlinear elastic
deformation \cite{GS_book_continuation}.
Our Shauder-type existence proof technique did not require that we establish
a contraction property for the fixed-point mapping $S$; this results in losing
access to a uniqueness result for the coupled system, as well losing access
to a fixed error reduction property for numerical methods based on the
fixed-point mapping $S$.

The coupling of elastic deformation to the electrostatic field is of
great importance in modeling the conformational change in large
macromolecules. To put this into perspective, more comprehensive and
realistic continuum models for macromolecular conformational changes
can be developed based on the results in this article, for example,
by coupling the (stochastic) hydrodynamical forces from the Stokes or
Navier-Stokes equation, or including van der Waals forces
between closely positioned molecules. While mathematical models and
robust numerical methods have been well studied for steady state
fluid-structure interaction problems \cite{Grandmont_3Dcouple}, the
inclusion of van der Waals forces appears to be more straightforward
\cite{QCui_05channelgating}. A major concern in applying these
coupled models, however, is the determination of the elasticity
properties of macromolecules within the continuum framework, which
requires new theoretical models and quantitative comparisons between
the continuum modeling and the classical molecular dynamical
simulation and/or experiential measurements. In a future work we
will study the development of numerical methods for this coupled
system and apply this model to macromolecular systems where
electrostatic forces play a dominant role.

\section{Acknowledgment}
The authors are grateful to Gary Huber, Ben-Zhuo Lu and Axel
Malqvist for discussions and/or a critical reading of the
manuscript. The work of Y.C.Z. and J.A.M. was supported in part by
the National Institutes of Health, the National Science Foundation,
the Howard Hughes Medical Institute, the National Biomedical
Computing Resource, the National Science Foundation Center for
Theoretical Biological Physics, the San Diego Supercomputing Center,
the W. M. Keck Foundation, and Accelrys, Inc. M.J.H was supported in
part by NSF Awards 0411723, 022560 and 0511766, in part by DOE
Awards DE-FG02-04ER25620 and DE-FG02-05ER25707, and in part by NIH
Award P41RR08605.

\bibliographystyle{abbrv}
\bibliography{../bib/books,../bib/papers,../bib/mjh,../bib/library,../bib/ref-gn,../bib/coupling,../bib/pnp}

\begin{thebibliography}{00}

\bibitem{Adams_SobolevBook} R. A. Adams, {\it Sobolev spaces},
Second edition, Pure and Applied Mathematics series, Vol 140, Academic Press, 2003.

\bibitem{Chen_rpbe} L. Chen, M. Holst and J. Xu,
The finite element approximation of the nonlinear Poisson-Boltzmann equation,
{\it SIAM J. Numer. Anal.}, 2007 (to appear).

\bibitem{HolstSaied_PBE} M. Holst and F. Saied,
Numerical solution of the Poisson-Boltzmann equation: Developing more robust and efficient
methods, {\it J. Comput. Chem.}, 16, 337-364, 1995.

\bibitem{HolstBakerWang_PBE} M. Holst, N. Baker and F. Wang,
Adaptive multilevel finite element solution of the Poisson-Boltzmann equation I: Algorithms and
examples, {\it J. Comput. Chem.}, 21, 1319-1342, 2000.

\bibitem{Thorpe_BiophyJ06} H. Gohlke and M. F. Thorpe, A natural coarse graining for simulating
large biomolecular motion, {\it Biophys. J.}, 91, 2115-2120, 2006.

\bibitem{Grandmont_3Dcouple} C. Grandmont, Existence of a three-dimensional steady state fluid-structure
interaction problem, {\it J. Math. Fluid Mech.}, 4, 76-94, 2002.

\bibitem{Ciarlet_book} P. G. Ciarlet, Mathematical Elasticity, Volume I: Three-dimensional elasticity,
{\it North-Holland}, 1988.

\bibitem{Weinberger_greenfunc} W. Littman, G. Stampacchia, and H. F. Weinberger, Regular points for elliptic equations with discontinuous coefficients,
{\it Annu. Scuola Norm. Sup. Pisa (3)}, 17, 43-77, 1963.

\bibitem{Tama_review06} F. Tama and C. L. Brooks, SYMMETRY, FORM, AND SHAPE: Guiding principles for robustness in macromolecular machines, {\it Annu. Rev. Biophys. Biomol. Struct.}, 35, 115-33, 2006.

\bibitem{Im_CPC98} W. Im and D. Beglov and B. Roux, Continuum solvation model: Electrostatic forces from numerical solutions to the Poisson-Bolztmann equation, {\it Comp. Phys. Comm.} 111, 59-75 (1998).

\bibitem{Gruter_greefunc} M. Gr\"uter and K.-O. Widman, The Green function for uniformly elliptic equations,
 {\it Manuscripta Math.}, 37, 303-342, 1982.

\bibitem{Moseenkov_sobolevcompos} V. B. Moseenkov, Composition of functions in Sobolev spaces,
 {\it Ukrainian Mathematical Journal}, 34, 316-319, 1982.

\bibitem{Gilson_Eforce} M. K. Gilson, M. E. Davis, B. A. Luty and
J. A. McCammon, Computation of electrostatic forces on solvated molecules
using the Poisson-Boltzmann equation, {\it J. Phys. Chem.}, 97, 3591-3600, 1993.

\bibitem{Davis_Eforce} M. E. Davis and J. A. McCammon,
Calculating electrostatic forces from grid-calculated potentials,
{\it J. Comput. Chem.}, 11, 401-409, 1990.

\bibitem{Yanyan_GradientEstiEllip} Y. Y. Li and M. Vogelius,
Gradient estimates for solution to divergence form elliptic equations with
discontinuous coefficients, {\it Arch. Rational Mech. Anal.}, 153, 91-151, 2000.

\bibitem{Yanyan_CompositeMaterial} Y. Y. Li and L. Nirenberg,
Estimates for elliptic systems from composite material,
{\it Comm. Pure Appl. Math.}, 56, 892-925, 2003.

\bibitem{Jerome_85} J. W. Jerome, Consistency of semiconductor modeling: an existence/stability
analysis for stationary van Roosbroeck system, {\it SIAM J. Appl. Math.}, 45, 565-590, 1985.

\bibitem{Cerutti_93} M. C. Cerutti, Integrability of reciprocals of the Green's function for
elliptic operator: counterexamples, {\it Proc. Amer. Math. Sco.}, 119, 125-134, 1993.

\bibitem{GT_pdebook} D. Gilbarg and N. S. Trudinger, Elliptic partial differential equations of second order,
Springer-Verlag, New York and Berlin, 1983.

\bibitem{Babuska_70} I. Babuska, The finite element method for elliptic equations with discontinuous coefficients,
{\it Computing}, 5, 207-213, 1970.

\bibitem{ChenZM_98} Z. Chen and J. Zou, Finite element methods and their convergence for elliptic
and parabolic interface problems, {\it Numerische Mathematik}, 79(2), 175-202, 1998.

\bibitem{ZouJ_02} J. Huang and J. Zou, Some new {\it a priori} estimates for second order
elliptic and parabolic interface problems, {\it J. Diff. Eq.}, 184, 570-586, 2002.

\bibitem{Seftel_64} Z. G. Seftel, The solution in $L^p$ and the classical solution of general boundary
value problems for elliptic equations with discontinuous coefficients (Russian), {\it Uspechi Math. Nauk},
19, 230-232, 1964.

\bibitem{Seftel_65} Z. G. Seftel, Energy inequalities and general boundary problems for elliptic
equations with discontinuous coefficients (Russian), {\it Sibirsk Math. Z.}, 6, 636-668, 1965.
19, 230-232, 1964.

\bibitem{Uchida_01} M. Uchida, Regularity of solutions of semilinear elliptic differential equations,
{\it J. Math. Sci. Univ. Tokyo}, 8, 357-363, 2001.

\bibitem{Baker_04} N. Baker, Poisson-Boltzmann methods for biomolecular electrostatics,
{\it Methods in Enzymology}, 383, 94-118, 2004.

\bibitem{Honig_90} K. Sharp and B. Honig, Electrostatic interactions in macromolecules: theory and
applications, {\it Annu. Rev. Biophys. Chem.}, 19, 301-332, 1990.

\bibitem{Holst_thesis} M. Holst, The Poisson-Boltzmann equation: analysis and Multilevel numerical
solution, {\it PhD thesis}, Numerical Computing Group, University of Illinois at Urbana-Champaign, 1994.

\bibitem{HXZhou_96} Z. Zhou, P. Payne, M. Vasquez, N. Kuhn and M. Levitt,
Finite-difference solution of the Poisson-Boltzmann equation: complete elimination of self-energy,
{\it J. Comput. Chem.}, 11, 1344-1351, 1996.

\bibitem{ZhouMIBPB} Y. C. Zhou, M. Feig and G. W. Wei, Highly accurate biomolecular electrostatics in continuum dielectric environments,
{\it J. Comput. Chem.}, In press.

\bibitem{Ben_bemPNAS} B. Lu, X. Cheng, J. Huang and J. A. McCammon,
Order N algorithm for computation of electrostatic interactions in biomolecular systems,
{\it Proceedings of the National Academy of Sciences of the United States of America},
103, 19314-19319, 2006.

\bibitem{Jones_MST} T. B. Jones, Basic theory of dielectrophoresis and electrorotation,
{\it IEEE Engineering in Medicine and Biology Magazine}, 22, 33-42, 2003.

\bibitem{Connolly_MS} M. L. Connolly, Analytical molecular surface calculation,
{\it J. Appl. Cryst.}, 16, 548-558, 1983.

\bibitem{QCui_05channelgating} Y. Tang, G. Cao, X. Chen, J. Yoo, A. Yethiraj and Q. Cui,
A finite element framework for studying the mechanical response
of macromolecules: application to the gating of the mechanosenstive channel MscL,
{\it Biophys. J.}, 91, 1248-1263, 2006.

\bibitem{Feng_membrane06} F. Feng and W. S. Klug, Finite element modeling of lipid bilayer membranes,
{\it J. Comput. Phys.}, 220, 394-408, 2006.

\bibitem{GS_book_continuation} E. I. Grigolyuk and V. I. Shalashilin, Problems of nonlinear deformation:
the continuation methods applied to nonlinear problems in solid mechanics, Kluwer Academic Publishers,
1990

\bibitem{Adcock_MD} S. A. Adcock, J.A. McCammon,
Molecular Dynamics: A Survey of Methods for Simulating the Activity of Proteins.
{\it Chem. Revs.}, 106, 1589-1615 (2006).

\bibitem{Zeid91a} E. Zeidler, Nonlinear Functional Analysis and its
Applications I: Fixed Point Theorems,
Springer-Verlag, New York and Berlin, 1991.

\bibitem{Sofonea_piezo04} M. Sofonea and EL-H. Essoufi, A
piezoelectric contact problem with slip dependent coefficient of
frition, {\it Mathematical Modelling and Analysis}, 9, 229-242,
2004.

\bibitem{Nisse_variation67}  E. R. EerNise,
Variational method for electroelastic vibration analysis,
{\it IEEE transactions on sonics and ultrasonics},
14, 153-159, 1967.

\bibitem{Han_piezo07} W. Han, M. Sofonea and K. Kazmi,
Analysis and numerical solution of a frictionless contact problem
for electro-elastic-visco-plastic materials, {\it Computer Methods
in Applied Mechanics and Engineering}, In press.

\bibitem{Xing_FS97} J. T. Xing and W. G. Price,
Variational Principles of Nonlinear Dynamical Fluid-Solid Interaction Systems,
{\it Phil. Trans. R. Soc. Lond. A}, 355, 1063-1095, 1997.

\end{thebibliography}


\vspace*{0.5cm}

\end{document}